\documentclass{article}[13pt]  
\usepackage[utf8x]{inputenc}
\usepackage{amsfonts} 
\usepackage{graphics}
\usepackage{graphicx} 
\usepackage{color}
\usepackage[margin=3.cm]{geometry}
\usepackage{xcolor}
\usepackage{amsmath, amsthm}
\newcommand{\dsp}{\displaystyle}
\newcommand{\eps}{\varepsilon} 
\newcommand{\curl}{\mbox{curl}}

\DeclareMathOperator {\curlb}{\bf{curl }}

\newcommand{\bv}[1]{{\bf{#1}}}

\newcommand{\ds}{\displaystyle}
\newtheorem{theorem}{Theorem}
\newtheorem{remark}{Remark}
\usepackage{amsfonts}
\usepackage{stmaryrd}
\usepackage{tabularx}
\newcolumntype{L}[1]{>{\raggedright\arraybackslash}p{#1}}
\newcolumntype{C}[1]{>{\centering\arraybackslash}p{#1}}
\newcolumntype{R}[1]{>{\raggedleft\arraybackslash}p{#1}}
\newcommand\T{\rule{0pt}{3ex}}       
\newcommand\B{\rule[-1.2ex]{0pt}{0pt}} 
\usepackage{graphicx} 
\DeclareGraphicsExtensions{.eps,.pdf,.png,.jpg,.JPG,.jpeg,.ps}
\graphicspath{{pics/}}
\allowdisplaybreaks

\usepackage{amsmath,amssymb}
\usepackage{xcolor}
\usepackage{appendix}
\usepackage{amsthm}
\usepackage{array}
\usepackage{mathtools}

\renewcommand{\epsilon}{\varepsilon}

\graphicspath{{figures/}}

\DeclareGraphicsExtensions{.eps,.pdf,.png,.jpg,.JPG,.jpeg,.ps,.pdf_tex}
\numberwithin{equation}{section}
\begin{document}


\title{A Fully Fourth Order Accurate Energy Stable {Finite Difference Method} for Maxwell's Equations in Metamaterials}
\author{Puttha Sakkaplangkul\footnote{Department of Mathematics, King Mongkut's Institute of Technology Ladkrabang, Ladkrabang, Bangkok 10520, Thailand}  \and  Vrushali Bokil\footnote{Department of Mathematics, Oregon State University,
  Corvallis, OR, 97331} \and  Camille Carvalho\footnote{Applied Mathematics Unit, School of Natural Sciences,
    University of California, Merced, 5200 North Lake Road, Merced, CA
    95343}}
   \maketitle

  \begin{abstract}
We present a novel fully fourth order in time and space {finite difference method for the time domain Maxwell's equations} in metamaterials. We consider a Drude metamaterial model for the material response to incident electromagnetic fields. We consider the second order formulation of the system of partial differential equations that govern the evolution in time of electric and magnetic fields along with the evolution of the polarization and magnetization current densities. Our discretization employs fourth order staggering in space of different field components and the modified equation approach to obtain fourth order accuracy in time. Using the energy method, we derive energy relations for the continuous models, and design numerical schemes that preserve a discrete analogue of the energy relation. Numerical simulations are provided in one and two dimensional settings to illustrate fourth order convergence as well as compare with second order schemes. 
  \end{abstract}

\\

\section{Introduction}
In this paper, we present the construction and analysis of fully fourth order in space and time {finite difference methods (FDMs)} based on staggered finite differences in space and the modified equation approach in time \cite{JoCo} for electromagnetic metamaterial models. Metamaterials have revealed a great opportunity for controlling light in nanophotonic devices, in particular by controlling surface plasmons that appear at the interfaces between the materials, with applications for antennas and optical cloaking (e.g. \cite{BaDeEb03,ZaSM05,cai2007,Maier07, sannomiya2008situ,GrBo10,akselrod2014probing}). Surface plasmons are sub-wavelength and highly sensitive to the interface's geometry. Therefore there is a need for accurate evaluations of the electromagnetic near fields that capture all the multiple scales inherent to those problems.

{The Yee-finite difference time domain (FDTD) method \cite{yee1966numerical} is one of the most well known FDMs for the numerical simulation of Maxwell's equations. Various second order accurate FDMs for metamaterial models have now been constructed} (e.g. \cite{ziol2001,mittra2008,jichun,pekmezci2014}). High-order schemes based on the first-order formulation of Maxwell's equations can be derived, but no associated discrete energy estimates have been proved {beyond second order in time \cite{bokil2014}.}

In this paper, we consider models based on the second order formulation of Maxwell's equations for electromagnetic wave propagation in materials that are characterized as Drude metamaterials \cite{jichun}. By converting temporal derivatives to spatial derivatives via the modified equation approach, we construct full fourth-order FDM schemes that mimic, at the discrete level, the energy estimate of the continuous models. Our construction is motivated by high order finite difference methods for the second order wave equation in the time domain as constructed in \cite{young1996,JoCo,cohenbook}. High-order numerical methods using second-order formulations have been recently used for dispersive media \cite{angel2019}. However, to our knowledge, discrete energy estimates have not been proved except for the scalar wave equation \cite{JoCo}. 

In Section \ref{sec:setting} of this paper we introduce the Maxwell-Drude model and its properties, Section \ref{sec:disc} presents the construction of the fourth-order FDM schemes, Section \ref{sec:stability} establishes stability of the semi and fully-discrete schemes in one spatial dimension. In Section \ref{sec:num} we present numerical results in one and two spatial dimensions, and we conclude in Section \ref{sec:conclu}. 

\section{Setting}\label{sec:setting}
\subsection{Drude Metamaterial Model: First Order Formulation}
We consider Maxwell's equations in a Drude metamaterial \cite{jichun}, given in the form 
\begin{equation} \label{EM_drude}
\begin{aligned}
& \eps_0 \partial_t \bv{E} + \bv{J}= \curlb  \bv{H}, \\
& \mu_0 \partial_t \bv{H} + \bv{K} = -\curlb \bv{E}, \\
& \partial_t \bv{J} =\eps_0 \omega_{pe}^2 \bv{E},\\
& \partial_t \bv{K} = \mu_0 \omega_{pm}^2\bv{H},\\
\end{aligned}
\end{equation} 
where $\bv{E}$ and $\bv{H}$ are the electric and magnetic fields, respectively, $\bv{J}$ and $\bv{K}$ are the polarization and magnetization current densities, respectively. The parameters $\eps_0$ and $\mu_0$ are the vacuum electric permittivity and magnetic permeability, respectively. Finally, the parameters $\omega_{pe}$ and $\omega_{pm}$ are the electric and magnetic plasma frequencies, respectively. System \eqref{EM_drude} has to be closed by adding initial and boundary conditions. 

System \eqref{EM_drude} is linear and thus admits harmonic plane wave solutions. Taking Fourier transforms in time, i.e. assuming dependence of solutions on $e^{-i \omega t}$, one finds that the Drude metamaterial model is characterized by a dispersive permittitivity and a permeability given as
\begin{equation}\label{eps_mu}
\eps(\omega) = \eps_0 \left( 1 - \frac{\omega_{pe}^2 }{\omega^2  }\right) , \quad \mu (\omega) = \mu_0 \left( 1 - \frac{\omega_{pm}^2 }{\omega^2 }\right),
\end{equation}
following Drude's law \cite{jichun}. We call system \eqref{EM_drude} the Maxwell-Drude model.

\subsection{Drude Metamaterial Model: Second Order Formulation}
System \eqref{EM_drude} can be rewritten in second order form in which we get two sets of decoupled equations, with one pair of equations involving only the field variables $(\bv{E}, \bv{K})$ modeled by the system of equations
\begin{subequations}
  \label{Drude_em1}
\begin{align}
  \label{Drude_em1.1}
& \partial_{tt} \bv{E}  + c^2 \curlb \curlb \bv{E} + {\omega_{pe}^2}  \bv{E} = -c^2 \curlb \bv{K},\\
  \label{Drude_em1.2}
& \partial_{tt} \bv{K}  + {\omega_{pm}^2}  \bv{K} = -\omega_{pm}^2 \curlb \bv{E},
\end{align}
\end{subequations}
where $c = 1/\sqrt{\varepsilon_0\mu_0}$, and the second pair involving the field variables $(\bv{H}, \bv{J})$ modeled by the system of equations
\begin{subequations}
  \label{Drude_em2}
\begin{align}
& \partial_{tt} \bv{H}  + c^2 \curlb \curlb \bv{H} + {\omega_{pm}^2}  \bv{H}  =c^2 \curlb \bv{J},\\
& \partial_{tt} \bv{J}  +\omega_{pe}^2 \bv{J} =  \omega_{pe}^2 \curlb \bv{H}.
\end{align}
\end{subequations} 
Since the two sub-systems in \eqref{Drude_em1} and \eqref{Drude_em2} are decoupled they can be solved separately. We note that in 2D and 3D, these systems will be coupled through the divergence conditions. In the 1D case, when the divergence conditions are decoupled from the curl equations, these systems are truly decoupled. In this paper, we do not consider the divergence conditions.
\subsection{Energy Estimates}
To derive energy estimates we need to define the following functional space:
\begin{equation}
  \begin{split}
    \bv{H}_{\#}(\curlb; \Omega) & : = \{ \mathbf{v} \in (L^2(\Omega))^3, \curlb \mathbf{v} \in (L^2(\Omega))^3, \\
    & \mathbf{v} \ \text{periodic on} \ \partial \Omega\}.
\end{split}
  \end{equation}  
From now on $\Vert \cdot \Vert$ denotes the $L^2$ norm in $\Omega$. From the second order formulation one can check that we have the following energy estimates.
\begin{theorem}\label{th:MD1}
  Assume $\bv{E} \in C^1((0,T); \bv{H}_{\#}(\curlb; \Omega))$, $\bv{K} \in C^1((0,T); \bv{H}_{\#}(\curlb;\Omega))$. System \eqref{Drude_em1} satisfies the following energy estimate
  \begin{equation}
    \frac{d\mathcal{E}_{EK}}{dt} = 0,
  \end{equation}
  where the energy $\mathcal{E}_{EK}$ is defined as 
\begin{equation}\label{eq:E_EK}
\begin{aligned}
\mathcal{E}_{EK} = \frac{1}{2}\Bigl( \frac{1}{c^2}||\partial_t \bv{E}||^2+  \frac{1}{\omega_{pm}^2} & ||\partial_t \bv{K}||^2 + \frac{\omega_{pe}^2}{c^2}||\bv{E}||^2  +||\curlb \bv{E} +\bv{K}||^2 \Bigr).
  \end{aligned}
\end{equation}
\end{theorem}

\begin{theorem}\label{th:MD2}
  Assume $\bv{H} \in C^1((0,T); \bv{H}_{\#}(\curlb; \Omega))$, $\bv{J} \in C^1((0,T); \bv{H}_{\#}(\curlb;\Omega))$. System \eqref{Drude_em2} satisfies the energy estimate
  \begin{equation}
    \frac{d\mathcal{E}_{HJ}}{dt} = 0,
  \end{equation}
  where the energy $\mathcal{E}_{HJ}$ is defined as 
\begin{equation}\label{eq:E_HJ}
\begin{aligned}
  \mathcal{E}_{HJ} = \frac{1}{2}\Bigl( \frac{1}{c^2}||\partial_t \bv{H}||^2+   \frac{1}{\omega_{pe}^2} & ||\partial_t \bv{J}||^2 +  \frac{\omega_{pm}^2}{c^2}||\bv{H}||^2  +||\curlb \bv{H} -\bv{J}||^2 \Bigr).
  \end{aligned}
\end{equation}
\end{theorem}
The proofs are based on the energy method which employs integration by parts. 

\section{Fourth order FDTD schemes} \label{sec:disc}
\subsection{Staggered High Order Spatial Finite Difference Methods} \label{sec:grid}

In this section, we describe the construction of second and fourth order approximations to the first and second order derivative
operators to get approximations to the vector $\curlb$ operator in 2D and 3D and the scalar $\curl$ operator in 1D and 2D. To that aim we need to define approximations of the derivative operators $\partial_x$, $\partial_y$, $\partial_z$. The construction presented in this section follows the exposition in \cite{cohenbook,bokil2011}. \\
{Consider $\Omega = [0,L]^d \subset  \mathbb{R}^d, \, d = 1,2,3,$ with $L\geq 0$. From now on we consider $d=3$.}
Given $M \in \mathbb{N}^\star := \mathbb{N} \setminus \lbrace 0 \rbrace$, let $h  = L/ M> 0$ be a uniform mesh step size, and define for $\ell = (i,j,k) \in \llbracket 0, M\rrbracket^3$, ${\bf{x}}_{\ell} =(x_i, y_j,z_k) =  (ih, jh,kh)$ \footnote{$\llbracket 0, M\rrbracket$ denotes the set of integers $0, \dots, M$.}. 
For $\ell = (i,j,k) \in \llbracket 0, M-1\rrbracket^3$, define $ {\bf{x}}^x_{\ell + \frac{1}{2}}=(x_{i+\frac{1}{2}}, y_j,z_k) $, $ {\bf{x}}^y_{\ell + \frac{1}{2}}=(x_i, y_{j+\frac{1}{2}},z_k) $, $ {\bf{x}}^z_{\ell + \frac{1}{2}}=(x_i, y_j,z_{k +\frac{1}{2}}) $. We define several staggered grids on $\Omega$. The {\it primal grid} in a direction along one of the axes, $G_p$, is defined as 
\begin{equation}
\label{eq:primalgrid}
G_p:  (X_\ell)_\ell, \quad  X \in \{x,y,z\}, \quad  l \in \llbracket 0, M\rrbracket .
\end{equation}
The {\it dual grid} in one direction, $G_d$, on $[0,L]$ is defined as  
\begin{equation}
\label{eq:dualgrid}
G_d: (X_{\ell+1/2})_\ell, \quad  X \in \{x,y,z\}, \quad  \ell \in \llbracket 0, M-1 \rrbracket .
\end{equation}
We define staggered grids on $[0,L]^3$ as ${\bf{G_p}} =G_p \times G_p \times G_p $, and ${\bf{G_d}}^x = G_d \times G_p \times G_p $,  ${\bf{G_d}}^y = G_p \times G_d \times G_p $, ${\bf{G_d}}^z = G_p \times G_p \times G_d $.
For smooth functions ${\bf{v}}$ and $\bv{u}$, define ${\bf{v}}_\ell \approx {\bf{v}}({\bf{x}}_{\ell})$, with ${\bf{x}_\ell \in {\bf{G_p}}}$, $\ell \in \llbracket 0, M\rrbracket^d$, and ${\bf{u}}^m_{\ell+\frac{1}{2}} \approx {\bf{u}}({\bf{x}}^m_{\ell+\frac{1}{2}})$, with ${\bf{x}}^m_{\ell + \frac{1}{2}} \in {\bf{G_d}}^m$, $m \in \{x,y,z\}$. On the primal grid $\bv{G_p}$ we define the discrete space  
\begin{equation}
  \begin{split}
{\bv{V}}_{0,h} & := \{\bv{v}_h:=(\bv{v}_\ell),\, \ell \in \llbracket 0, M\rrbracket^d,  \,  {||\bv{v}_h||_{0,h}:=} h\sum_{\ell = 0}^{M}|\bv{v}_\ell|^2 < \infty\}  + \ \text{periodic B.C.}
\end{split}
  \end{equation}
and on the dual grids $\bv{G_d}^m$, $m = x,y,z$, we define the discrete spaces
\begin{equation}
\begin{aligned}
& \bv{V}^m_{\frac{1}{2},h}  := \{ \bv{u}^m_h :=\left( \bv{u}^m_{\ell+\frac{1}{2}} \right),  \, \ell \in \llbracket 0, M-1 \rrbracket^d, {||\bv{u}^m_h||_{\frac{1}{2},h,m} : =}  h\sum_{\ell =0}^{M-1}|\bv{u}^m_{\ell+\frac{1}{2}} |^2 < \infty \}  +  \ \text{periodic B.C}.
  \end{aligned}
\end{equation}
The $\ell^2$ norms on ${\bv{V}}_{0,h}$, and $\bv{V}^m_{\frac{1}{2},h} $, denoted by $||\cdot||_{0,h}$, and $||\cdot||_{\frac{1}{2},h,m}$, respectively,  are 
derived from corresponding $\ell^2$ scalar products $\langle \cdot,\cdot \rangle_{0,h}$, and $\langle \cdot,\cdot\rangle_{\frac{1}{2},h,m}$.
We now define the discrete second order finite difference operators 
\begin{subequations}
\label{op2s}
\begin{align}
\begin{split}
& \mathcal{F}^{(2)}_{h,m}:{\bv{V}}_{0,h} \rightarrow \bv{V}^m_{\frac{1}{2},h} , \ \forall \bv{v}_h \in {\bv{V}}_{0,h} ,\,  \left(\mathcal{F}^{(2)}_{h,m} \bv{v}_h\right)_{\ell+\frac{1}{2}} 
  :=
  \ds\frac{\bv{v}^m_{\ell+1}-\bv{v}_{\ell}}{h},\label{ops1}
  \end{split}\\
   \begin{split}
&
\mathcal{F}^{(2)}_{h,m \ast}:\bv{V}^m_{\frac{1}{2},h}  \rightarrow \bv{V}_{0,h}, \ \forall \bv{u}^m_h \in \bv{V}^m_{\frac{1}{2},h} , \, \left(\mathcal{F}^{(2)}_{h,m \ast} \bv{u}^m_h\right)_{\ell} 
	:=
  \ds\frac{\bv{u}^m_{\ell+\frac{1}{2}}-\bv{u}^m_{\ell-\frac{1}{2}} }{h},
   \label{ops2}
  \end{split}
\end{align}
\end{subequations}
and the discrete fourth order finite difference operators,
\begin{subequations}
\label{op4s}
\begin{align}
\begin{split}
& \mathcal{F}^{(4)}_{h,m}:{\bv{V}}_{0,h} \rightarrow \bv{V}^m_{\frac{1}{2},h} , \  \forall \bv{v}_h \in {\bv{V}}_{0,h} ,\, \left(\mathcal{F}^{(4)}_{h,m} \bv{v}_h\right)_{\ell+\frac{1}{2}} :=
\ds \frac{9}{8} \frac{\bv{v}^m_{\ell+1} -\bv{v}_{\ell} }{h} - \frac{1}{24} \frac{\bv{v}^m_{\ell+2}- \bv{v}^m_{\ell-1} }{h}, \\
  \end{split} \\
   \begin{split}
&\mathcal{F}^{(4)}_{h,m \ast}:\bv{V}^m_{\frac{1}{2},h}  \rightarrow \bv{V}_{0,h}, \   \forall \bv{u}^m_h \in \bv{V}^m_{\frac{1}{2},h} ,\, \left(\mathcal{F}^{(4)}_{h,m \ast} \bv{u}^m_h\right)_{\ell} :=
\ds\frac{9}{8} \frac{\bv{u}^m_{\ell+\frac{1}{2}}-\bv{u}^m_{\ell-\frac{1}{2}} }{h} - \frac{1}{24} \frac{\bv{u}^m_{\ell+\frac{3}{2}} - \bv{u}^m_{\ell-\frac{3}{2}} }{h}.
  \end{split}
\end{align}
\end{subequations}
Operators \eqref{op2s}-\eqref{op4s} are the second-order and fourth-order discrete approximations of the operator $\partial_m$, $m \in \{x,y,z\}$ with step size $h$, respectively.
Finally we can define the second and fourth order discrete approximations of the 3D $\curlb$ operator {and its dual} in terms of $ \mathcal{F}^{(2)}_{h,m}$, and $\mathcal{F}^{(4)}_{h,m}$, $m \in \{x,y,z\}$, {and $ \mathcal{F}^{(2)}_{h,m\ast}$, and $\mathcal{F}^{(4)}_{h,m\ast}$, respectively}. For all $\bv{v}_h = (v_h^x, v_h^y,v_h^z)$, $\bv{u}_h = (u_h^x, u_h^y,u_h^z)$, and for $p = 2,4$, we have  
\hspace*{-15pt}\begin{subequations}
\label{ops}
\begin{align}
& \begin{split}
\hspace*{-25pt}\left(\curlb^{(p)}_h\bv{v}_h\right)_{\ell+\frac{1}{2}} := & \left(  \left(\mathcal{F}^{(p)}_{h,y} {v}^z_h\right)_{\ell+\frac{1}{2}}- \left(\mathcal{F}^{(p)}_{h,z} {v}^y_h\right)_{\ell+\frac{1}{2}},\right.  \left(\mathcal{F}^{(p)}_{h,z} {v}^x_h\right)_{\ell+\frac{1}{2}}- \left(\mathcal{F}^{(p)}_{h,x} {v}^z_h\right)_{\ell+\frac{1}{2}}, \left. \left(\mathcal{F}^{(p)}_{h,x} {v}^y_h\right)_{\ell+\frac{1}{2}}- \left(\mathcal{F}^{(p)}_{h,y} {v}^x_h\right)_{\ell+\frac{1}{2}} \right), 
 \end{split} \\ 
&  \begin{split}
\left(\curlb^{(p)}_{h,\ast} \bv{u}_h\right)_{\ell} :=& \left( \left(\mathcal{F}^{(p)}_{h,y\ast} {u}^z_h\right)_{\ell}- \left(\mathcal{F}^{(p)}_{h,z\ast} {u}^y_h\right)_{\ell}, \right.      \left(\mathcal{F}^{(p)}_{h,z\ast} {u}^x_h\right)_{\ell}- \left(\mathcal{F}^{(p)}_{h,x\ast} {u}^z_h\right)_{\ell},   \left. \left(\mathcal{F}^{(p)}_{h,x\ast} {u}^y_h\right)_{\ell}- \left(\mathcal{F}^{(p)}_{h,y\ast} {u}^x_h\right)_{\ell} \right).
 \end{split}
\end{align}
\end{subequations}
Note that for $d = 2$ one needs vector and scalar curl operators. For $p = 2,4$, their discrete version, respectively,  is defined as
\begin{subequations}
	\label{discurl}
	\begin{align}
&	\left(\curlb_h^{(p)} {v}_h \right)_{\ell+\frac{1}{2}} := \left( \left( \mathcal{F}_{h,y}^{(p)}v_h \right)_{\ell+\frac{1}{2}} , - \left( \mathcal{F}_{h,x}^{(p)}v_h\right)_{\ell+\frac{1}{2}}   \right), \\
	& \left(\curl_h^{(p)} \bv{v}_h \right)_{\ell+\frac{1}{2}} := \left(\mathcal{F}_{h,x}^{(p)}v_h^y \right)_{\ell+\frac{1}{2}} -  \left( \mathcal{F}_{h,y}^{(p)}v_h^x\right)_{\ell+\frac{1}{2}}  ,
	\end{align}
\end{subequations}
and similarly for $\curlb_{h,\ast}^{(p)} $,  $\curl_{h,\ast}^{(p)} $. Finally for $d=1$, the curl operator and its dual, up to a sign, reduce to $ \mathcal{F}^{(p)}_{h,m}$ and $ \mathcal{F}^{(p)}_{h,m \ast}$, $p=2,4$, $m \in \{x,y,z\}$, respectively.
\subsection{Semi-discretization in Space}
We start with the spatial discretization of the system for the pair $(\bv{E},\bv{K})$ and then follow that with the discretization for the second sub-system corresponding to the pair $(\bv{H},\bv{J})$. The discretizations are based on the staggered finite difference methods presented above on dual and primal grids. 

Using the discretization introduced in section \ref{sec:grid}, we obtain the following semi-discretizations: for $\bv{E}_h$ defined on a primal grid, 
and $\bv{K}_h$ defined on a dual grid,
(similarly for $\bv{H}_h$ defined on a primal grid,
and $\bv{J}_h$ defined on a dual grid),
\begin{subequations}
  \label{MD_semi}
\begin{align}
 \begin{split}
    & \partial_{tt} \bv{E}_h  + c^2 (\curlb^{(4)}_{h,\ast} \curlb^{(4)}_{h})\bv{E}_h + {\omega_{pe}^2}  \bv{E}_h  = -c^2 \curlb^{(4)}_{h,\ast} \bv{K}_h,\\
    \end{split}\\
& \partial_{tt} \bv{K}_h  + {\omega_{pm}^2}  \bv{K}_h = -\omega_{pm}^2 \curlb^{(4)}_{h}  \bv{E}_h,\\
\begin{split}
  & \partial_{tt} \bv{H}_h  + c^2 (\curlb^{(4)}_{h,\ast}  \curlb^{(4)}_{h}) \bv{H}_h + {\omega_{pm}^2}  \bv{H}_h  =c^2 \curlb^{(4)}_{h,\ast}  \bv{J}_h,\\
  \end{split}\\
& \partial_{tt} \bv{J}_h  +\omega_{pe}^2 \bv{J}_h =  \omega_{pe}^2 \curlb^{(4)}_{h}\bv{H}_h .
\end{align}
\end{subequations} 
\subsection{Modified Equation Approach for Time Discretization}
Given $N \in \mathbb{N}^\star$, we define $\Delta t = T/N$, and $t^n = n \Delta t$, $n \in \llbracket 0, N\rrbracket$. For all $\bv{X}_\ell^n = \bv{X}(\bv{x}_\ell, t^n)$, we have 
\begin{equation}\label{eq:time_disc}
\begin{aligned}
\partial_{tt} \bv{X}_\ell^n  &= \dsp \frac{\bv{X}_\ell^{n+1} - 2 \bv{X}_\ell^n + \bv{X}_\ell^{n-1}}{\Delta t^2}  - \frac{\Delta t^2}{12} \partial^4_{t} \bv{X}_\ell^n + O(\Delta t^4) :=\delta_{tt} \bv{X}_\ell^n - \frac{\Delta t^2}{12} \partial^4_{t} \bv{X}_\ell^n + O(\Delta t^4).\\
\end{aligned}
\end{equation}

Based on \cite{JoCo}, we use the modified equation approach to convert fourth order time derivatives into derivatives in space. If the spatial derivatives are then discretized by second order finite differences, we will be able to achieve overall fourth order accuracy in space and time. Starting with \eqref{Drude_em1.1}, and finding an expression of $\partial^4_t \bv{E}$ in terms of spatial derivatives, assuming we can interchange time and space differentiation and using \eqref{Drude_em1} we have  
\[\begin{aligned}
 \partial^4_{t} \bv{E}  & = -  c^2 \curlb \curlb \partial_{tt}\bv{E}- {\omega_{pe}^2}  \partial_{tt}\bv{E} -c^2 \curlb\partial_{tt}\bv{K} ,\\
  = & -  c^2 \curlb  \curlb  (-  c^2 \curlb  \curlb\bv{E} - {\omega_{pe}^2}  \bv{E} -c^2 \curlb \bv{K}) \\
   - & {\omega_{pe}^2}  (-  c^2 \curlb \curlb\bv{E}- {\omega_{pe}^2}  \bv{E} -c^2 \curlb \bv{K} ) \\
 - &c^2 \curlb (- {\omega_{pm}^2}  \bv{K} -\omega_{pm}^2 \curlb \bv{E}) ,\\
\end{aligned}
\]
leading to 
\begin{equation} \label{eq:conversion_MD_E}
\begin{aligned}
 &\partial^4_{t} \bv{E}   = c^4 (\curlb  \curlb)^2  \bv{E} +  c^2 (2 {\omega_{pe}^2}+{\omega_{pm}^2})  \curlb  \curlb\bv{E} \\
& +{\omega_{pe}^4}  \bv{E}  
  + c^4  \curlb \curlb \curlb\bv{K} +c^2 ( {\omega_{pe}^2}+{\omega_{pm}^2})  \curlb \bv{K} .
 \end{aligned}
\end{equation}
Finally the discrete version of \eqref{eq:conversion_MD_E} becomes:
\begin{subequations} \label{Disc_conver_MD}
\begin{align}\label{disc_MD_E}
&  \partial^4_{t} \bv{E}_h   =  c^4  (\curlb^{(2)}_{h,\ast}  \curlb^{(2)}_{h})^2   \bv{E}_h  +  c^2 (2 {\omega_{pe}^2}+{\omega_{pm}^2}) (\curlb^{(2)}_{h,\ast}  \curlb^{(2)}_{h})\bv{E}_h +{\omega_{pe}^4}  \bv{E}_h \\
 &  + c^4  (\curlb^{(2)}_{h} \curlb^{(2)}_{h,\ast}  \curlb^{(2)}_{h})\bv{K}_h +c^2 ( {\omega_{pe}^2}+{\omega_{pm}^2})\curlb^{(2)}_{h,\ast} \bv{K}_h.  \notag
\end{align}
Similarly, we obtain the following equations for the other variables
\begin{align}\label{disc_MD_K}
 &\partial^4_{t} \bv{K}_h   = c^2 \omega_{pm}^2  (\curlb^{(2)}_{h}  \curlb^{(2)}_{h,\ast} \curlb^{(2)}_{h} )  \bv{E}_h + \omega_{pm}^2 ( {\omega_{pe}^2} +{\omega_{pm}^2}) \curlb^{(2)}_{h}\bv{E}_h \\
  &+ c^2 {\omega_{pm}^2}   (\curlb^{(2)}_{h}  \curlb^{(2)}_{h,\ast})  \bv{K}_h   + \omega_{pm}^4 \bv{K}_h,  \notag
\end{align}
\begin{align}\label{disc_MD_H}
& \partial^4_{t} \bv{H}_h  =  c^4  (\curlb^{(2)}_{h,\ast}  \curlb^{(2)}_{h})^2   \bv{H}_h + c^2 ( 2{\omega_{pm}^2}+{\omega_{pe}^2})  (\curlb^{(2)}_{h,\ast}  \curlb^{(2)}_{h})\bv{H}_h +{\omega_{pm}^4}  \bv{H}_h  \\
 &-c^4  (\curlb^{(2)}_{h} \curlb^{(2)}_{h,\ast}  \curlb^{(2)}_{h})\bv{J}_h -c^2 ( {\omega_{pm}^2}+{\omega_{pe}^2})\curlb^{(2)}_{h,\ast}\bv{J}_h, \notag
 \end{align}
\begin{align}\label{disc_MD_J} 
&  \partial^4_{t} \bv{J}_h   =  - c^2 \omega_{pe}^2  (\curlb^{(2)}_{h}  \curlb^{(2)}_{h,\ast} \curlb^{(2)}_{h} )  \bv{H}_h  -\omega_{pe}^2 ( {\omega_{pm}^2}+{\omega_{pe}^2}) \curlb^{(2)}_{h}\bv{H}_h \\
  &+ c^2 {\omega_{pe}^2} (\curlb^{(2)}_{h}  \curlb^{(2)}_{h,\ast}) \bv{J}_h   + \omega_{pe}^4 \bv{J}_h .\notag
\end{align}
\end{subequations}

\begin{remark}
For this model, the schemes follow exactly the idea in \cite{JoCo}. {The difference from \cite{JoCo} is that we have systems of PDEs composed of wave-like equations, along with a dissipative term (not present in \cite{JoCo})}. 
\end{remark}
\subsection{Full discretization}
Substituting the modified equations \eqref{Disc_conver_MD} into \eqref{eq:time_disc} for each field, then into the semi-discrete equations \eqref{MD_semi} we obtain the fully fourth order discrete scheme
\begin{subequations}
  \label{MD_full}
\begin{align}
\begin{split} 
&  \dsp  \delta_{tt} \bv{E}_\ell^n+ c^2  (\curlb^{(4)}_{h,\ast} \curlb^{(4)}_{h})\bv{E}_\ell^n + {\omega_{pe}^2}  \bv{E}_\ell^n + c^2 \curlb^{(4)}_{h,\ast}  \bv{K}_{\ell+1/2}^n  - \frac{\Delta t^2}{12}  \left[ c^4  (\curlb^{(2)}_{h,\ast}  \curlb^{(2)}_{h})^2   \bv{E}_\ell^n \right. \\
& \left. +  c^2 (2 {\omega_{pe}^2}+{\omega_{pm}^2})  (\curlb^{(2)}_{h,\ast}  \curlb^{(2)}_{h})\bv{E}_\ell^n +{\omega_{pe}^4}  \bv{E}_\ell^n  \right.  \left. + c^4  (\curlb^{(2)}_{h} \curlb^{(2)}_{h,\ast}  \curlb^{(2)}_{h})\bv{K}_{\ell+1/2}^n  \right. \\
& \left. +c^2 ( {\omega_{pe}^2}+{\omega_{pm}^2})\curlb^{(2)}_{h,\ast} \bv{K}_{\ell+1/2}^n  \right] =0,  \label{MD_full_a}\\
\end{split}  &  \\
  \begin{split} 
&  \dsp  \delta_{tt} \bv{K}_{\ell + 1/2}^n + {\omega_{pm}^2}  \bv{K}_{\ell+1/2}^n + \omega_{pm}^2 \curlb^{(4)}_{h}  \bv{E}_\ell^n - \frac{\Delta t^2}{12}  \left[ c^2 \omega_{pm}^2 (\curlb^{(2)}_{h}  \curlb^{(2)}_{h,\ast} \curlb^{(2)}_{h} ) \bv{E}_\ell^n \right. \\
& \left. + \omega_{pm}^2 ( {\omega_{pe}^2}+{\omega_{pm}^2}) \curlb^{(2)}_{h}\bv{E}_\ell^n \right.  \left. + c^2 {\omega_{pm}^2}  (\curlb^{(2)}_{h}  \curlb^{(2)}_{h,\ast})  \bv{K}_{\ell+1/2}^n   + \omega_{pm}^4 \bv{K}_{\ell+1/2}^n \right] =0, \label{MD_full_b}\\
\end{split} &\\
  \begin{split} 
&  \dsp  \delta_{tt} \bv{H}_\ell^n + c^2 (\curlb^{(4)}_{h,\ast} \curlb^{(4)}_{h}) \bv{H}_\ell^n  + {\omega_{pm}^2}  \bv{H}_\ell^n   -c^2 \curlb^{(4)}_{h,\ast} \bv{J}_{\ell+1/2}^n  - \frac{\Delta t^2}{12}  \left[  c^4 (\curlb^{(2)}_{h,\ast}  \curlb^{(2)}_{h})^2   \bv{H}_\ell^n \right. \\
& \left. + c^2 ( 2{\omega_{pm}^2}+{\omega_{pe}^2}) (\curlb^{(2)}_{h,\ast}  \curlb^{(2)}_{h})\bv{H}_\ell^n +{\omega_{pm}^4}  \bv{H}_\ell^n \right. \\
& \left.   -c^4 \ (\curlb^{(2)}_{h} \curlb^{(2)}_{h,\ast}  \curlb^{(2)}_{h})\bv{J}_{\ell+1/2}^n \right.  \left. -c^2 ( {\omega_{pm}^2}+{\omega_{pe}^2})\curlb^{(2)}_{h,\ast} \bv{J}_{\ell+1/2}^n\right] =0, \\
\end{split} &\\
  \begin{split} 
&  \dsp  \delta_{tt} \bv{J}_{\ell + 1/2}^n+\omega_{pe}^2 \bv{J}_{\ell+1/2}^n -\omega_{pe}^2 \curlb^{(4)}_{h}  \bv{H}_\ell^n - \frac{\Delta t^2}{12}  \left[ - c^2 \omega_{pe}^2  (\curlb^{(2)}_{h}  \curlb^{(2)}_{h,\ast} \curlb^{(2)}_{h} )  \bv{H}_\ell^n  \right. \\
& \left. -\omega_{pe}^2 ( {\omega_{pm}^2}+{\omega_{pe}^2})\curlb^{(2)}_{h} \bv{H}_\ell^n \right.  \left.  + c^2 {\omega_{pe}^2}   (\curlb^{(2)}_{h}  \curlb^{(2)}_{h,\ast}) \bv{J}_{\ell+1/2}^n   + \omega_{pe}^4 \bv{J}_{\ell+1/2}^n \right] =0, \\
\end{split} &
\end{align}
\end{subequations}
where $\displaystyle \dsp  \delta_{tt}$ is the three point time stencil defined in \eqref{eq:time_disc}. The discrete system \eqref{MD_full} is then fourth order in time and space. We will refer to it as a (4,4)-scheme. 
\begin{remark}
In the numerical examples we will compare with the so-called $(2,2)$-FDM scheme and $(2,4)$-FDM scheme. The $(2,4)$-scheme is obtained from  \eqref{MD_full} by dropping the terms multiplied by $\frac{\Delta t ^2}{12}$, while the $(2,2)$-scheme is obtained from  \eqref{MD_full} by dropping the terms multiplied by $\frac{\Delta t ^2}{12}$ and replacing the fourth order discrete operators by second order discrete operators {defined in \eqref{ops1} and \eqref{ops2}.}
\end{remark}
\vspace{-0.3cm}
\begin{remark}
Contrary to Yee like schemes, a specific class of FDTD methods that stagger both in time and space, our approach uses only staggering in space.
\end{remark}
\section{Stability {and Discrete Energy Analysis} (in 1D)}\label{sec:stability}
For simplicity, we detail the energy estimates for the semi-discrete and fully-discrete schemes in one dimension, and for the pair $(\bv{E}, \bv{K})$. Similar results hold for the other pair, and similarly for 2D and 3D. We now consider generically $E,K$, and the scalar $\curl$ operator. 
\subsection{Semi-discrete energy estimates}
The semi-discrete system for $(E,K)$ in one dimension satisfies the following energy result:
\begin{theorem}[Semi-discrete Energy Estimate]
  Assume $E_h \in C^1((0,T); V_{0,h}(\Omega))$, $K_h \in C^1((0,T); V_{\frac{1}{2},h}(\Omega))$. The spatially discrete system \eqref{MD_semi} satisfies the {\it energy conservation} result
  \begin{equation}
    \frac{d\mathcal{E}_{h,EK}}{dt} = 0,
  \end{equation}
  where the energy $\mathcal{E}_{h,EK}$ is defined as 
\begin{equation}
\begin{aligned}
  \mathcal{E}_{h,EK} 
  =&  \frac{1}{2}\left(  \frac{1}{c^2}||\partial_t E_h||_{0,h}^2 
  +\frac{1}{\omega_{pm}^2}||\partial_t K_h||_{\frac{1}{2},h}^2 \right.  + \frac{\omega_{pe}^2}{c^2}||E_h||_{0,h}^2 
   \left. +||K_h+ \mathrm{curl}_h^{(4)}E_h||_{\frac{1}{2},h}^2 \right).
  \end{aligned}
\end{equation}
\end{theorem}
For simplicity we denote {the discrete norm $\Vert \cdot \Vert_{\frac{1}{2},h,m} = \Vert \cdot \Vert_{\frac{1}{2},h} $}, without making $m$ precise: $m$ will be fixed by the considered polarization.
The proof is based on the fact that under {the assumption of periodic boundary conditions,} the operator $-\curl_{h,\ast}^{(4)}$ is the adjoint of the discrete operator $\curl_h^{(4)}$ with respect to the  $\ell^2$ scalar product. 
\subsection{Fully-Discrete Energy Estimates} 
The full fourth-order discretization of the one dimensional version of \eqref{MD_full_a}-\eqref{MD_full_b} can be rewritten in vector form as
	\begin{align}
	\ds \label{MD_full_mat2}
	\mathcal{P} \delta_{tt} \bv{W}_j^n 
	+ \mathcal{A}_1 \bv{W}_j^n 
	- \frac{\Delta t^2 c^2}{12} \mathcal{A}_2 \bv{W}_j^n = \textbf{0},
	\end{align}

	\noindent where $\bv{W}_j^n = (E_j^n, K_{j+\frac{1}{2}}^n)^t$, and the matrix operators $\mathcal{P}, \mathcal{A}_1$ and $ \mathcal{A}_2$ are defined by

		\begin{subequations}
			\begin{align}
						&\mathcal{A}_1 = 
			\begin{bmatrix} \ds
			- \curl_{h,\ast}^{(4)} \curl_{h}^{(4)} + \frac{\omega_{pe}^2}{c^2} 
			& - \curl_{h,\ast}^{(4)}  \\
			\curl_{h}^{(4)}& 1
			\end{bmatrix},	
			&\mathcal{P} = 
			\begin{bmatrix} \ds
			1/c^2 & 0 \\ 0 & \ds 1/\omega_{pm}^2
			\end{bmatrix},
			\, 
			\mathcal{A}_2
			=  \begin{bmatrix} a_{11} & a_{12} \\ a_{21} & a_{22}			 
			\end{bmatrix},
						\end{align}
		\end{subequations}
		with 
	\[ \begin{aligned}
	& a_{11} := \dsp 	(\curl_{h,\ast}^{(2)} \curl_{h}^{(2)} )^2- \left( 2 \omega_{pe}^2 + \omega_{pm}^2\right) \frac{\curl_{h,\ast}^{(2)} \curl_{h}^{(2)} }{c^2} + \frac{\omega_{pe}^4}{c^4},\\
	& a_{12} := \dsp	\curl_{h,\ast}^{(2)} \curl_{h}^{(2)} \curl_{h,\ast}^{(2)}  - \left( \omega_{pe}^2 + \omega_{pm}^2\right)\frac{\curl_{h,\ast}^{(2)} }{c^2},   \\
	&{
	a_{21} :=	\ds 
	-\curl_{h}^{(2)} \curl_{h,\ast}^{(2)} \curl_{h}^{(2)} 
	+  \left( \omega_{pe}^2 + \omega_{pm}^2\right)\frac{\curl_{h}^{(2)}}{c^2}  },  \\
	& {
	a_{22} := \ds 
	-\curl_{h}^{(2)} \curl_{h,\ast}^{(2)}  + \frac{\omega_{pm}^2}{c^2}. } 
	\end{aligned} 
	\]

\noindent We define discrete inner product and norm as follows. For any $\bv{\widehat{W}}^n_j = (\widehat{E}^n_j,\widehat{K}^n_{j+\frac{1}{2}})^t,\,
\bv{\widetilde{W}}^n_j = (\widetilde{E}^n_j,\widetilde{K}^n_{j+\frac{1}{2}})^t $, $n \in \llbracket 0, N\rrbracket$, $j \in \llbracket 0, M-1\rrbracket$, with a uniform mesh step size $h$, $ \langle \bv{\widehat{W}}^n, \bv{\widetilde{W}}^n \rangle_h  :=  h \sum_{j=0}^{M-1} 	\langle \bv{\widehat{W}}^n_j, \bv{\widetilde{W}}^n_j \rangle  = \langle \widehat{E}^n, \widetilde{E}^n \rangle_{0,h} + \langle \widehat{K}^n, \widetilde{K}^n \rangle_{\frac{1}{2},h}$, and $\Vert\bv{\widehat{W}}^n \Vert^2_h := \Vert \widehat{E}^n \Vert_{0,h}^2 + \Vert \widehat{K}^n \Vert_{\frac{1}{2},h}^2$.
\begin{theorem}[Fully-discrete Energy Estimate] Assuming periodic boundary conditions for all function fields on $\Omega$,
 the fully-discrete system satisfies the estimate
  \begin{equation} \label{en1}
   \delta_t \mathcal{E}_h^n =\frac{1}{\Delta t} \left[
	\mathcal{E}_h^{n+\frac{1}{2}}
	-
	\mathcal{E}_h^{n-\frac{1}{2}}
	\right]= 0,
  \end{equation}
  where the discrete quantity $\mathcal{E}_h^{n+\frac{1}{2}}$ is defined as 
\begin{equation}
\begin{aligned}
\hspace*{-10pt}\mathcal{E}_h^{n+\frac{1}{2}} 
	&= 
	\frac{1}{2}
	\bigg[
	\frac{1}{c^2}\Vert \delta_t E^{n+\frac{1}{2}} \Vert_{0,h}^2
	+ \frac{1}{\omega_{pm}^2}\Vert \delta_t K^{n+\frac{1}{2}} \Vert_{\frac{1}{2},h}^2   + \langle \mathcal{A}_1 \bv{W}^{n+1}, \bv{W}^{n} \rangle_h
	- \frac{\Delta t^2 c^2}{12} \langle \mathcal{A}_2\bv{W}^{n+1}, \bv{W}^{n} \rangle_h
	\bigg],  \label{eng1}
  \end{aligned}
\end{equation}
with 
$\bv{W}^{n}  = (E^n,K^n)^t$.
Under the  {Courant-Friedrichs-Lewy (CFL) condition $\dsp \frac{ c \Delta t }{h} <1$, we have that $\mathcal{E}_h^{n+\frac{1}{2}}  \geq 0$, and is then called a {\it discrete energy}, which is conserved in time. }
\end{theorem}
	\begin{proof}
	\noindent Multiply \eqref{MD_full_mat2} by  $\ds 
	\widetilde{\bv{W}_j^{n} } = \frac{1}{2\Delta t} \left( \bv{W}_j^{n+1} - \bv{W}_j^{n-1} \right) $,
	\noindent we get 
\begin{align}
 \label{MD_full_mat3}
\begin{split}
\langle  \mathcal{P}  \delta_{tt} \bv{W}^n , \widetilde{\bv{W}^{n} } \rangle_h
	+ \langle \mathcal{A}_1 \bv{W}^n , & \widetilde{\bv{W}^{n} }  \rangle_h\\
	- \frac{\Delta t^2 c^2}{12}  & \langle \mathcal{A}_2 \bv{W}^n,  \widetilde{\bv{W}^{n} } \rangle_h = 0.
\end{split}	\end{align}

	\noindent It is then straight forward to show that 
	\begin{equation}
		\label{eq1}
		\begin{aligned}
		\ds 
		\langle \mathcal{P}\ \delta_{tt} & \bv{W}^n ,  \widetilde{\bv{W}^{n} } \rangle_h
		= \frac{1}{2\Delta t}
		\left(
		\Vert \mathcal{P} \delta_{t} \bv{W}^{n+\frac{1}{2}}  \Vert^2_h 
		-
		\Vert \mathcal{P} \delta_{t} \bv{W}^{n-\frac{1}{2}}  \Vert^2_h 
		\right) , \\
		&= \frac{1}{2\Delta t}
		\left(
		\frac{1}{c^2}\Vert \delta_t E^{n+\frac{1}{2}} \Vert_{0,h}^2
		+ \frac{1}{\omega_{pm}^2}\Vert \delta_t K^{n+\frac{1}{2}} \Vert_{\frac{1}{2},h}^2 \right.  \left. - \frac{1}{c^2}\Vert \delta_t E^{n-\frac{1}{2}} \Vert_{0,h}^2
		- \frac{1}{\omega_{pm}^2}\Vert \delta_t K^{n-\frac{1}{2}} \Vert_{\frac{1}{2},h}^2
		\right).
		\end{aligned}
	\end{equation}
		
	\noindent As both operators $\mathcal{A}_1$ and $\mathcal{A}_2$ are self-adjoint, the following identities hold: for $i=1,2$
	\begin{subequations}
		\label{eq2}
		\begin{align*}
		\ds
		\langle \mathcal{A}_i\bv{W}^n ,  & \widetilde{\bv{W}^n} \rangle_h =   \frac{1}{2\Delta t} 
		\left( 
		\langle \mathcal{A}_i\bv{W}^{n+1}, \bv{W}^{n} \rangle_h
		-
		\langle \mathcal{A}_i\bv{W}^n, \bv{W}^{n-1} \rangle_h 
		\right) .
		\end{align*}
	\end{subequations}

	Then one obtains the result by substituting into \eqref{MD_full_mat3}.
	\end{proof}

\section{Numerical results}\label{sec:num}
\subsection{Results in 1D}
In this section, we consider $\bv{E} = (0,E,0)$, and $\bv{K} = (0,0,K)$, with invariance of the fields with respect to $y$ and $z$. The second order Maxwell-Drude model reduces to the system
\begin{equation} \label{EM1D_EK}
\begin{aligned}
& \partial_{tt}E - c^2 \partial_{xx}E + {\omega_{pe}^2}  E=  c^2 \partial_x K, \\
& \partial_{tt} K  +\omega_{pm}^2  K = -\omega_{pm}^2 \partial_x E .\\
\end{aligned}
\end{equation} 
We consider the domain $\Omega = [0, 1]$, {and in the numerical examples we set} $\varepsilon_0 = 5$, $\mu_0 = 0.2$. We define an exact solution for \eqref{EM1D_EK} (with periodic boundary conditions) of the form
\begin{subequations}
 \label{exact_sol}
\begin{align} \ds
E(x,t) &= \frac{1}{\omega} \sin(\omega \pi t) \sin(k \pi x),  \\
K(x,t) &= \frac{\mu_0\omega_{pm}^2}{\pi \omega} \sin(\omega \pi t) \cos(k \pi x).
\end{align}
\end{subequations}

\noindent {Such solution satisfies the Maxwell-Drude model in particular for the parameters $k = 2,\, \omega_{pe} = 26.63199$ and
\begin{subequations}
	\begin{align} \ds
		& \hspace*{-9pt}\omega_{pm} = \sqrt{\frac{1}{\mu_0} \left( \frac{\omega_{pe}^2}{c^2(\varepsilon_0 - k)} - k\pi^2 \right)}, \, \omega = \frac{\omega_{pe}}{\pi}\sqrt{\frac{\varepsilon_0}{\varepsilon_0 - k}}.
	\end{align}
\end{subequations}  }
 \noindent  The largest time step and mesh sizes are taken to be $\Delta t = 0.02$ and $h = \Delta x = c \Delta t/\nu$, respectively, where $\nu < 1$ is the  Courant number in the CFL condition. These values are
 successively decreased by half to obtain the corresponding rates in six simulation runs.

 \subsubsection{Numerical Computation of the Discrete Energy}
 
Substituting the exact solution \eqref{exact_sol} into \eqref{eq:E_EK} one obtains
\begin{equation}
	\label{Exact_eng}
\begin{aligned}
 \hspace*{-11pt} \mathcal{E}_{EK} =  &\frac{1}{2}
\bigg( \frac{1}{c^2} \Vert\partial_t E  \Vert^2
+\frac{1}{\omega_{pm}^2} \Vert \partial_t K  \Vert^2 + \frac{\omega_{pe}^2}{c^2}\Vert E \Vert^2 
+ \Vert  K  + \partial_x E \Vert^2 
\bigg) = 13.30148848500039.
\end{aligned}
\end{equation}

\noindent  We will compare our new (4,4)-scheme with  the (2,4)-scheme and the (2,2)-scheme. To validate our method we will compare for each scheme the computed discrete energy with \eqref{Exact_eng}. We define the {Relative Energy Error} as {the error in the discrete energy with respect to the continuous energy as}{
\begin{align} \label{Rel_eng} \ds
{	
\Theta^{n+\frac{1}{2}}_{ij} = \displaystyle \left|\frac{ \mathcal{E}^{n+ \frac{1}{2}}_{h, ij} - \mathcal{E}_{EK} }{\mathcal{E}_{EK} } \right|, \quad ij \in \lbrace 22, 24, 44 \rbrace .}
\end{align}
}
 \noindent
 The discrete energy of (4,4)-scheme, $	\mathcal{E}^{n+\frac{1}{2}}_{h,44} $, is defined in \eqref{eng1}, while for the (2,4)- and (2,2)- schemes we have 
 \begin{subequations}
 	\begin{align} \ds
 	\mathcal{E}^{n+\frac{1}{2}}_{h,24} 
 	= 	\frac{1}{2}&	
 	\Big(
 	\frac{1}{c^2}\Vert \delta_t E^{n+\frac{1}{2}} \Vert_{0,h}^2	+ \frac{1}{\omega_{pm}^2}\Vert \delta_t K^{n+\frac{1}{2}} \Vert_{\frac{1}{2},h}^2   +\langle \mathcal{A}_1 \bv{W}^{n+1}, \bv{W}^{n} \rangle_h 	\Big),   \\
 	%
 	\mathcal{E}^{n+\frac{1}{2}}_{h,22} 
 	= 	\frac{1}{2}&
 	\Bigl(
 	\frac{1}{c^2}\Vert \delta_t E^{n+\frac{1}{2}} \Vert_{0,h}^2
 	+ \frac{1}{\omega_{pm}^2}\Vert \delta_t K^{n+\frac{1}{2}} \Vert_{\frac{1}{2},h}^2 	+ \Big \langle \mathcal{F}^{(2)}_{h} E^{n+1} + K^{n+1} , \mathcal{F}^{(2)}_{h} E^n + K^n  \Big\rangle_{\frac{1}{2},h}   	+ \frac{\omega_{pe}^2}{c^2} \langle E^{n+1}, E^n   \rangle_{0,h}    
 	\Bigr).
 	\end{align}
 \end{subequations}
 
Table \ref{Table1} provides the Relative Energy Errors for the three schemes, and for various CFL numbers $\nu$. Results validate the energy conservation.

 \begin{table} [h!]
 	\centering
 	\begin{tabular}{ |C{1.cm}|C{2.cm}|C{2.cm}|C{2.cm}| }
 		\hline
 		\multicolumn{4}{|c|}{Relative Energy Error, {$\Theta^{n+\frac{1}{2}}_{ij}$}, for $\Delta t = 0.02$} \T \B  \\
 		\hline
 		\textbf{$\nu$}   & (4,4)-scheme & (2,4)-scheme & (2,2)-scheme  \T \B   \\
 		\hline  
 		0.2 & 7.8372e-16  &  6.2698e-16  & 7.8372e-16 \\
 		\hline
 		0.5 & 7.8372e-16  & 4.7023e-16  &6.2698e-16 \\
 		\hline 
 		0.8 &  9.4047e-16
 		  & 3.1349e-16
 		    &  4.7023e-16 \\
 		\hline   
 		\multicolumn{4}{|c|}{Relative Energy Error, {$\Theta^{n+\frac{1}{2}}_{ij}$}, for $\Delta t = 0.01$} \T \B  \\
 		\hline
 		\textbf{$\nu$}   & (4,4)-scheme & (2,4)-scheme & (2,2)-scheme  \T \B   \\
 		\hline  
 		0.2 & 1.3893e-15  & 1.2504e-15  & 8.3361e-16 \\
 		\hline
 		0.5 & 1.1115e-15  & 1.5283e-15  & 8.3361e-16 \\
 		\hline 
 		0.8 &  1.6672e-15
 		& 1.2504e-15
 		&  8.3361e-16 \\
 		\hline  
	\end{tabular} 
 	\vspace{.3cm}
 	\caption{Relative Energy Error, {$\Theta^{n+\frac{1}{2}}_{ij}$, defined in \eqref{Rel_eng}} for the 1D Maxwell-Drude Metamaterial.}
 	\label{Table1}
 	 \vspace{-0.5cm}
 \end{table}	

  \subsubsection{Convergence Rate}
 
 We evaluate the convergence rate of the (4,4)-scheme by computing the discrete errors:
 \begin{subequations}
 	\label{conrate}
 \begin{align}
 \ds
 \text{Err(E)} &= \max_{0 \leq n \leq N}\Vert E^n - E(\cdot, t^n) \Vert_{0,h}, \\
  \text{Err(K)} &= \max_{0 \leq n \leq N}\Vert K^n- K(\cdot,t^n) \Vert_{\frac{1}{2},h}.
 \end{align}
 \end{subequations} 
\noindent Figure \ref{fig1}, Tables \ref{Table2} and \ref{Table3} show that the (4,4)-scheme is indeed fourth order for the pair $(E,K)$ whereas the (2,4)-scheme and (2,2)-scheme are second order, { for the CFL $\nu =2$. Note that the rates are computed with respect to $M$, not $\Delta x$. Similar results hold for other Courant numbers $\nu$.}
	\begin{table}[h!]
	\centering
\begin{tabular}{ |C{2cm}|C{2cm}|C{2cm}|C{2cm}| }
		\hline
		\multicolumn{4}{|c|}{  {Err(E)}  with CFL $\nu = 0.2$} \T \B  \\
		\hline
		$\Delta t$ & $\Delta x$   & \multicolumn{2}{c|}{(4,4)-scheme}    \T \B   \\
		\cline{3-4}
		& & \textbf{Error}  & \textbf{Rate} \\
		\hline		 
		2.000e-02 & 1.000e-01    
		          & 6.280e-04      & - \\
		\hline    
		1.000e-02 & 5.000e-02   
					& 3.730e-05      & -4.073 \\
		\hline 
		5.000e-03 & 2.500e-02   
					& 2.303e-06     & -4.018 \\
		\hline		
		2.500e-03 & 1.250e-02   
					& 1.435e-07      & -4.005 \\
		\hline	
		1.250e-03 & 6.250e-03   
					& 8.959e-09      & -4.001 \\
		\hline	
		6.250e-04 & 3.125e-03  
					& 5.600e-10      & -4.000\\
		\hline 			
	\end{tabular} 
	\vfill
\begin{tabular}{ |C{2cm}|C{2cm}|C{2cm}|C{2cm}| }
		\hline
		\multicolumn{4}{|c|}{  {Err(E)}  with CFL $\nu = 0.2$} \T \B  \\
		\hline
		$\Delta t$ & $\Delta x$   & \multicolumn{2}{c|}{(2,4)-scheme}    \T \B   \\
		\cline{3-4}
		& & \textbf{Error}  & \textbf{Rate} \\
		\hline		 
		2.000e-02 & 1.000e-01    
		         & 4.383e-02     & - \\
		\hline    
		1.000e-02 & 5.000e-02   
					& 1.102e-02      & -1.991\\
		\hline 
		5.000e-03 & 2.500e-02   
					& 2.722e-03     & -2.018 \\
		\hline		
		2.500e-03 & 1.250e-02   
					& 6.777e-04      & -2.006\\
		\hline	
		1.250e-03 & 6.250e-03   
					& 1.693e-04      & -2.002\\
		\hline	
		6.250e-04 & 3.125e-03  
					& 4.230e-05      & -2.000\\
		\hline 			
	\end{tabular} 
		\vfill
\begin{tabular}{ |C{2cm}|C{2cm}|C{2cm}|C{2cm}| }
		\hline
		\multicolumn{4}{|c|}{  {Err(E)}  with CFL $\nu = 0.2$} \T \B  \\
		\hline
		$\Delta t$ & $\Delta x$   & \multicolumn{2}{c|}{(2,2)-scheme}    \T \B   \\
		\cline{3-4}
		& & \textbf{Error}  & \textbf{Rate} \\
		\hline		 
		2.000e-02 & 1.000e-01    
		           & 4.070e-02 & -  \\
		\hline    
		1.000e-02 & 5.000e-02   
				 & 1.026e-02      & -1.989  \\
		\hline 
		5.000e-03 & 2.500e-02   
					& 2.538e-03     & -2.015  \\
		\hline		
		2.500e-03 & 1.250e-02   
					& 6.324e-04      & -2.005  \\
		\hline	
		1.250e-03 & 6.250e-03   
					& 1.580e-04      & -2.001  \\
		\hline	
		6.250e-04 & 3.125e-03  
					& 3.948e-05      & -2.000  \\			
		\hline 			
	\end{tabular} 
	\caption{Errors of $E$ for the Maxwell-Drude metamaterial in 1D.}
	\label{Table2}
\end{table}	
 
  	\begin{table}[h!]
 	\centering
 \begin{tabular}{ |C{2cm}|C{2cm}|C{2cm}|C{2cm}| }
 		\hline
 		\multicolumn{4}{|c|}{ {Err(K)} with CFL $\nu = 0.2$} \T \B  \\
 		\hline
 		$\Delta t$ & $\Delta x$   & \multicolumn{2}{c|}{(4,4)-scheme}   \T \B   \\
 		\cline{3-4}
 		& & \textbf{Error}  & \textbf{Rate} \\
 		\hline		 
 		2.000e-02 & 1.000e-01   
 		& 4.360e-02      & - \\
 		\hline    
 		1.000e-02 & 5.000e-02   
 		& 2.591e-03      & -4.073\\
 		\hline 
 		5.000e-03 & 2.500e-02   
 		& 1.599e-04     & -4.018 \\
 		\hline		
 		2.500e-03 & 1.250e-02   
 		& 9.965e-06      & -4.005 \\
 		\hline	
 		1.250e-03 & 6.250e-03   
 		& 6.223e-07      & -4.001 \\
 		\hline	
 		6.250e-04 & 3.125e-03  
 		& 3.890e-08      & -4.000 \\
 		\hline	
 	\end{tabular} 
 	\vfill
\begin{tabular}{ |C{2cm}|C{2cm}|C{2cm}|C{2cm}| }
 		\hline
 		\multicolumn{4}{|c|}{ {Err(K)} with CFL $\nu = 0.2$} \T \B  \\
 		\hline
 		$\Delta t$ & $\Delta x$   & \multicolumn{2}{c|}{(2,4)-scheme}   \T \B   \\
 		\cline{3-4}
 		& & \textbf{Error}  & \textbf{Rate} \\
 		\hline		 
 		2.000e-02 & 1.000e-01   
 		& 3.026e+00      & -\\
 		\hline    
 		1.000e-02 & 5.000e-02   
 		& 7.604e-01      & -1.993\\
 		\hline 
 		5.000e-03 & 2.500e-02   
 		& 1.877e-01     & -2.018 \\
 		\hline		
 		2.500e-03 & 1.250e-02   
 		& 4.674e-02      & -2.006\\
 		\hline	
 		1.250e-03 & 6.250e-03   
 		& 1.167e-02      & -2.002\\
 		\hline	
 		6.250e-04 & 3.125e-03  
 		& 2.918e-03      & -2.000\\
 		\hline	
 	\end{tabular} 
 	\vfill
\begin{tabular}{ |C{2cm}|C{2cm}|C{2cm}|C{2cm}| }
 		\hline
 		\multicolumn{4}{|c|}{ {Err(K)} with CFL $\nu = 0.2$} \T \B  \\
 		\hline
 		$\Delta t$ & $\Delta x$   & \multicolumn{2}{c|}{(2,2)-scheme}   \T \B   \\
 		\cline{3-4}
 		& & \textbf{Error}  & \textbf{Rate} \\
 		\hline		 
 		2.000e-02 & 1.000e-01   
 					& 2.880e+00      & -  \\
 		\hline    
 		1.000e-02 & 5.000e-02   
 					& 7.153e-01      & -2.010  \\
 		\hline 
 		5.000e-03 & 2.500e-02   	
 					& 1.763e-01     & -2.020  \\
 		\hline		
 		2.500e-03 & 1.250e-02   
 					& 4.390e-02      & -2.006  \\
 		\hline	
 		1.250e-03 & 6.250e-03   
			 		& 1.096e-02      & -2.002  \\
 		\hline	
 		6.250e-04 & 3.125e-03  
 					& 2.740e-03      & -2.000  \\	
 		\hline	
 	\end{tabular} 
 	\vfill
 	\caption{Errors of $K$ for the Maxwell-Drude metamaterial in 1D.}
 	\label{Table3}
 \end{table}	
 \begin{figure}[h!]
 	\centering
 	\includegraphics[scale= 0.22]{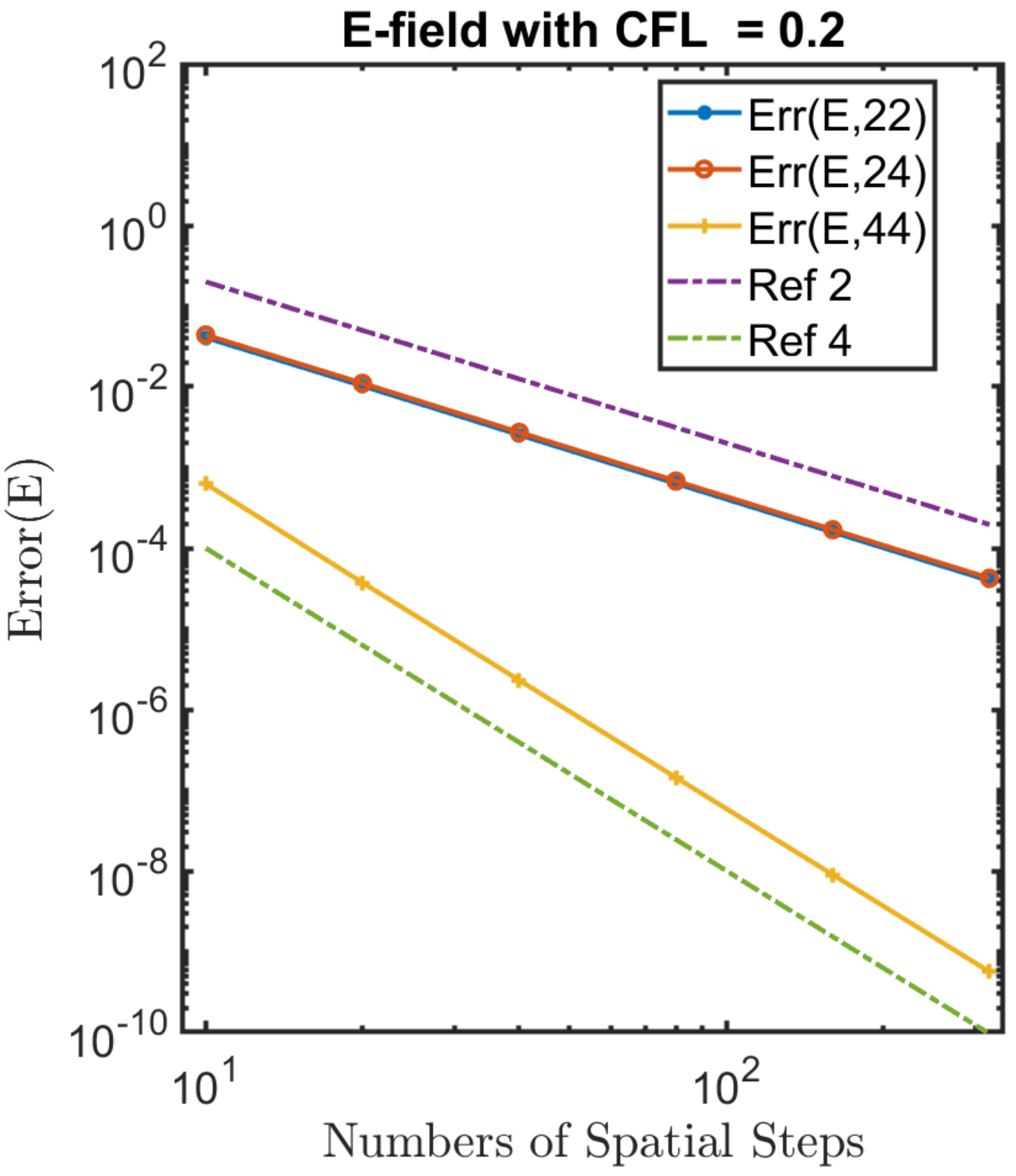} 
 	 	\includegraphics[scale= 0.22]{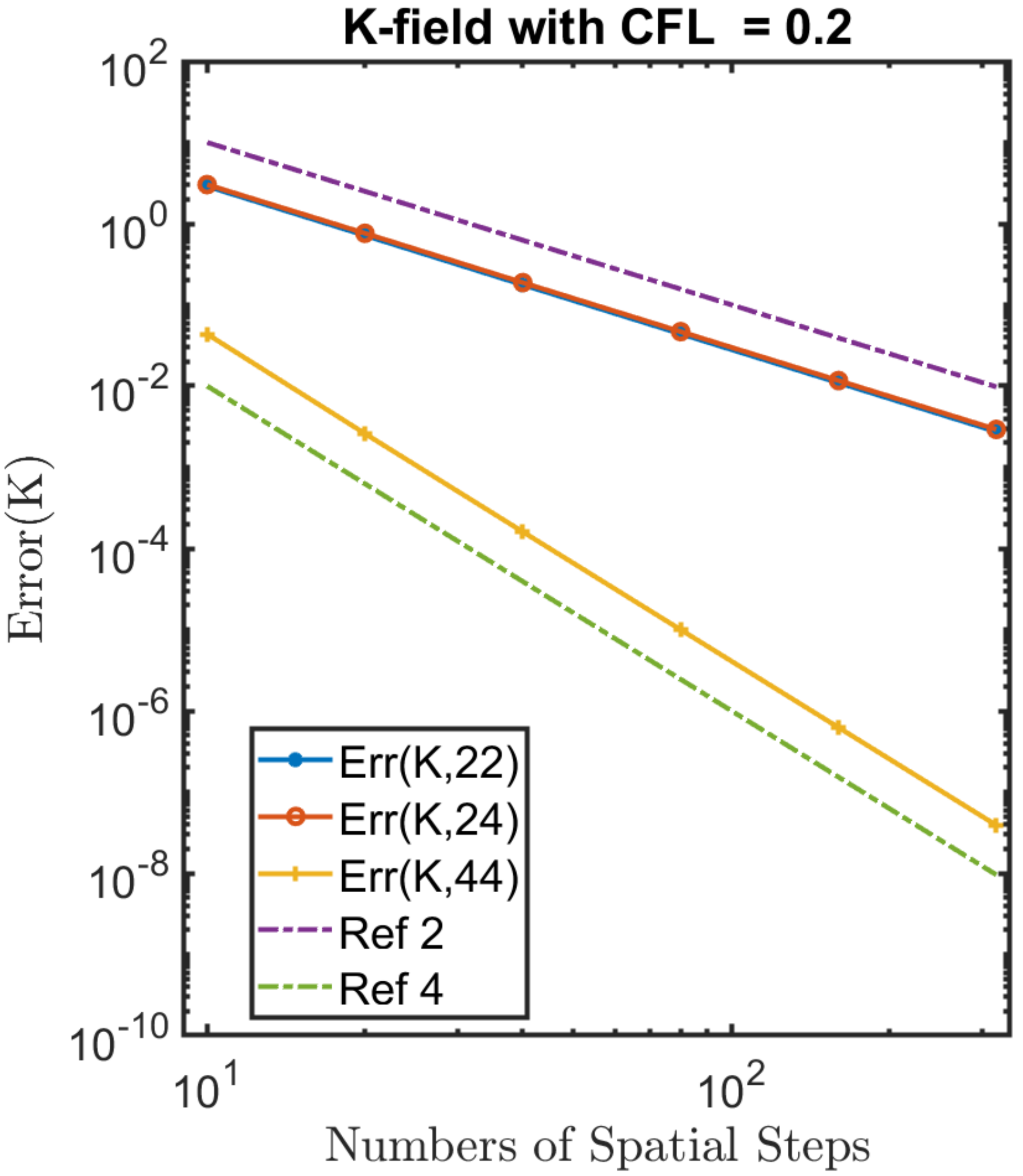} 
 	\caption{ Convergence rates for 1D the Maxwell-Drude metamaterial: for $E$ (left), for $K$ (right).}	
 	\label{fig1}
 \end{figure}
 
 \subsubsection{Long time computation}
We investigate the stability of each scheme over time. We set $\nu= 0.2$, and perform two experiments:
 \begin{align*}
 	\text{Case 1: } & T = 250 \text{ and } \Delta t = 0.02 \text{, then } \Delta x = 0.1 ,\\
 	\text{Case 2: } & T = 50 \text{ and } \Delta t = 0.004 \text{, then }  \Delta x = 0.02 .
 \end{align*}
 In each case, and for each scheme, we compute the relative energy errors.
Figure \ref{fig5} shows the obtained results. The relative energy errors are controlled under $10^{-14}$ after a long time, which is in accordance with the energy conservation. Note that $(4,4)$-scheme provides better results. {The plots of differences in energy $\ds \mathcal{E}^{n+\frac{1}{2}}_h - \mathcal{E}^{n-\frac{1}{2}}_h$ are also analyzed and they are essentially zero, which also confirms our theoretical analysis of the energy estimate \eqref{en1} and the results in Table \ref{Table1}.}
 \begin{figure}[h!]
 	\centering
 	\includegraphics[scale= 0.35]{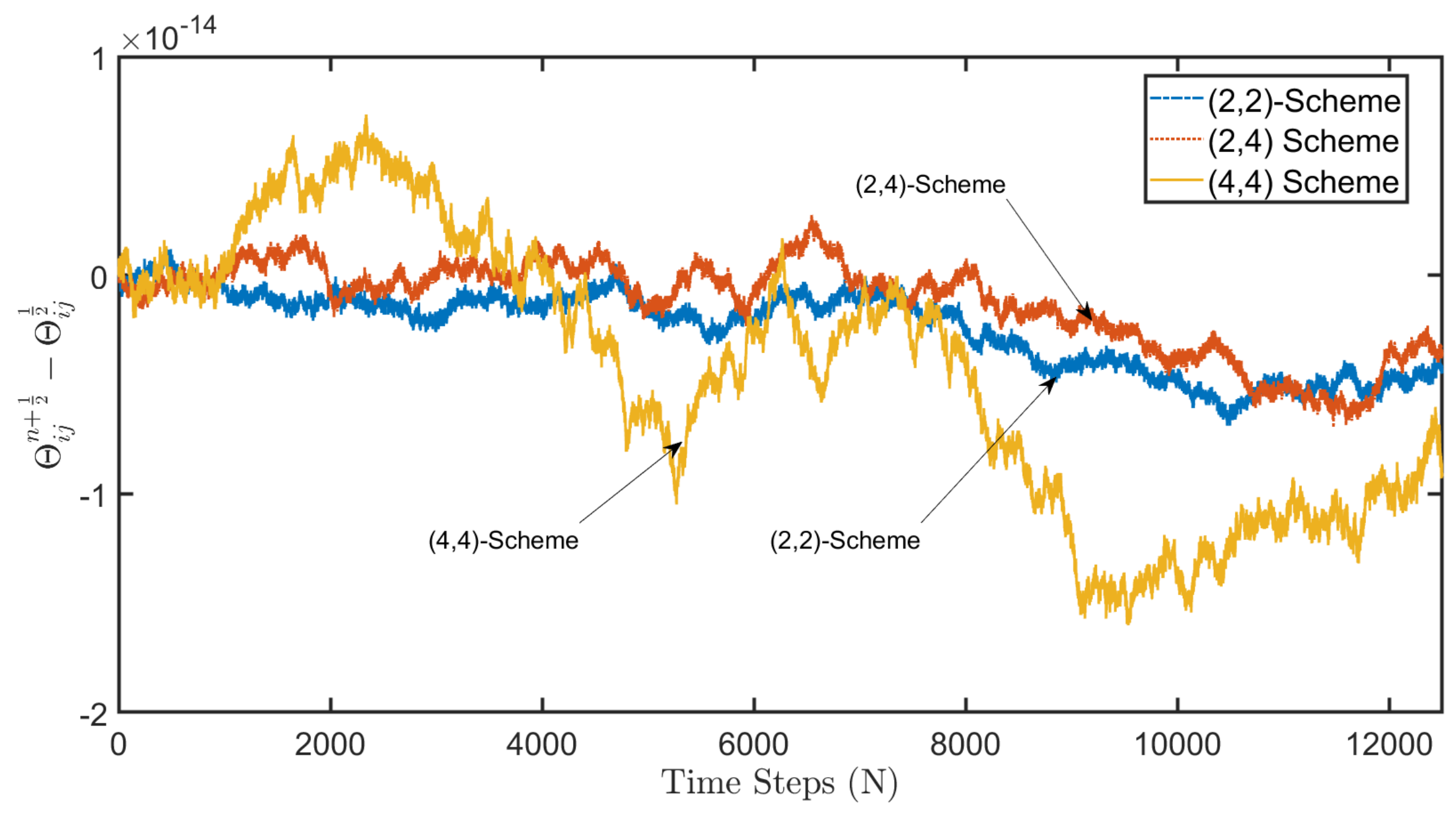} 
	\includegraphics[scale= 0.35]{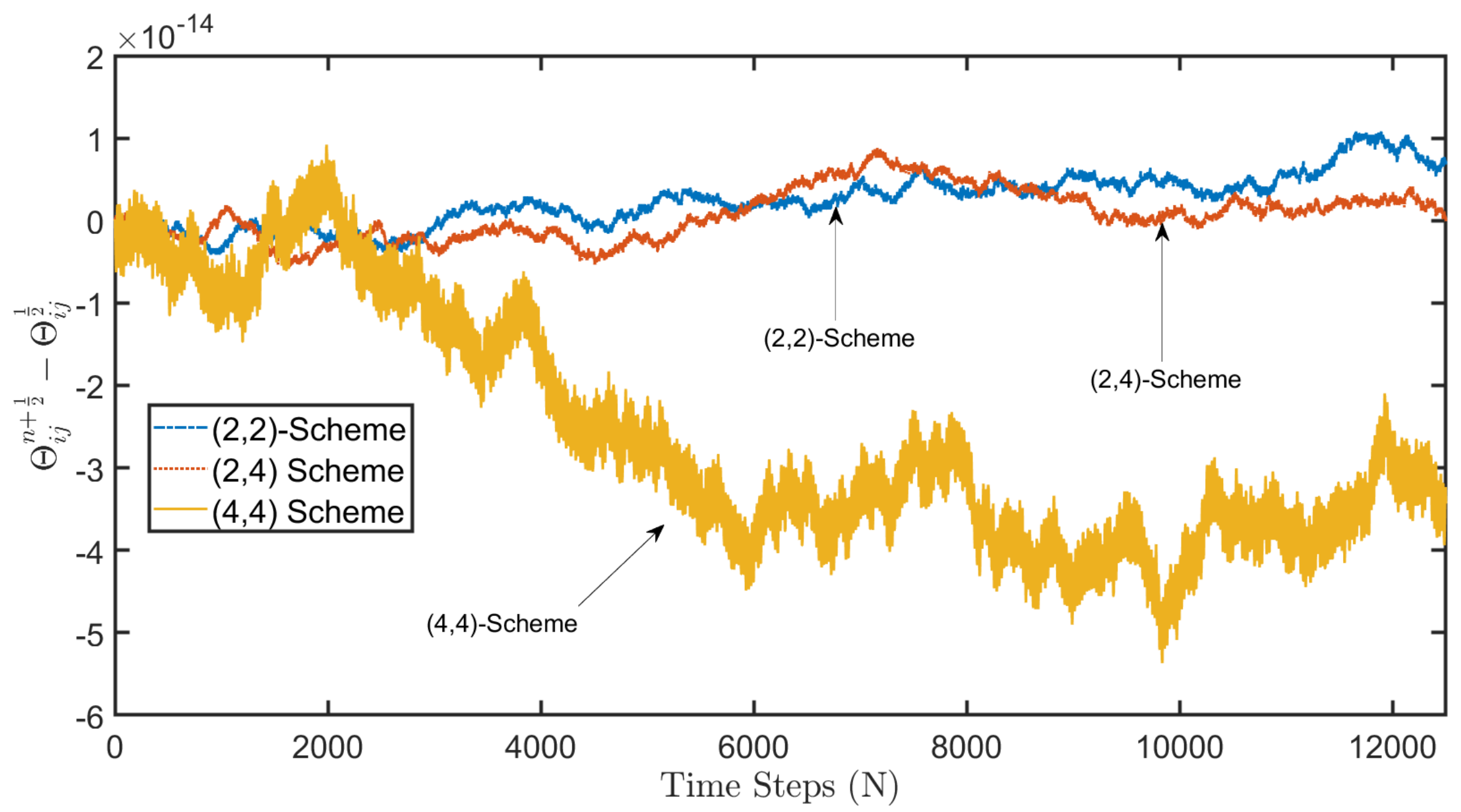} 
 	\caption{Differences in Relative energy errors {$\Theta_{ij}^{n+\frac{1}{2}}-\Theta_{ij}^{\frac{1}{2}}$} over time: Case 1 (Top); Case 2 (Bottom).} 
 	\label{fig5}
 \end{figure}
\vspace{-0.5cm}

\subsection{Results in 2D}
We consider now the two dimensional Maxwell-Drude model for the transverse electric (TE) polarization: $\bv{E} = (E_x,E_y)$ and $K = K_z$. The second order formulation for the $(\bv{E}, K)$ pair becomes:
	\begin{subequations}
\begin{align}
& \partial_{tt} \bv{E}  + c^2 \curlb \curl \,\bv{E} + {\omega_{pe}^2}  \bv{E} = -c^2 \curlb {K},\\
& \partial_{tt} {K}  + {\omega_{pm}^2}  {K} = -\omega_{pm}^2 \curl \,\bv{E},
\end{align}
	\end{subequations}
	along with periodic conditions. One obtains the $(4,4)$-scheme for the two dimensional Maxwell-Drude model from \eqref{MD_full_a}-\eqref{MD_full_b} by plugging the polarization and using the discrete operators \eqref{discurl}.
We consider the domain $\Omega = [0,1] \times[0,1]$, and design an exact solution of the form
	\begin{subequations}
		\label{exact_solD2}
		\begin{align} \ds
		E_x(x,y,t) 
		&= -\frac{k_y}{\omega} \sin(\omega \pi t) \sin(k_x \pi x)  \cos(k_y \pi y),  \\
		E_y(x,y,t) 
		&= \frac{k_x}{\omega} \sin(\omega \pi t) \cos(k_x \pi x)  \sin(k_y \pi y),  \\
		K(x,y ,t) &= \frac{\mu_0\omega_{pm}^2}{\pi \omega} \sin(\omega \pi t) \sin(k_x \pi x)  \sin(k_y \pi y),
		\end{align}
	\end{subequations}
 and we set $\bv{k} = (k_x,k_y) = (2,2)$, $\omega_{pe} = 10$,
	\begin{subequations}
		\begin{align*}
		\ds 
 \omega = \frac{\omega_{pe}}{\pi}\sqrt{\frac{\varepsilon_0}{1+\varepsilon_0}}, \,	\omega_{pm} = \left[ \frac{1}{\mu_0}\left( |\bv{k}|^2\pi^2 + \frac{\omega_{pe}^2}{c^2(1+\varepsilon_0)} \right) \right]^{1/2}.
		\end{align*}
	\end{subequations}
	The largest time step and mesh sizes are taken similarly to the 1D case. The numerical and exact solutions for a Maxwell-Drude metamaterial at $t = 0.5$ are shown in Figure \ref{fig3} for $K$. 
	   \begin{figure}[h!]
	  	\centering
	  	\includegraphics[scale= 0.45]{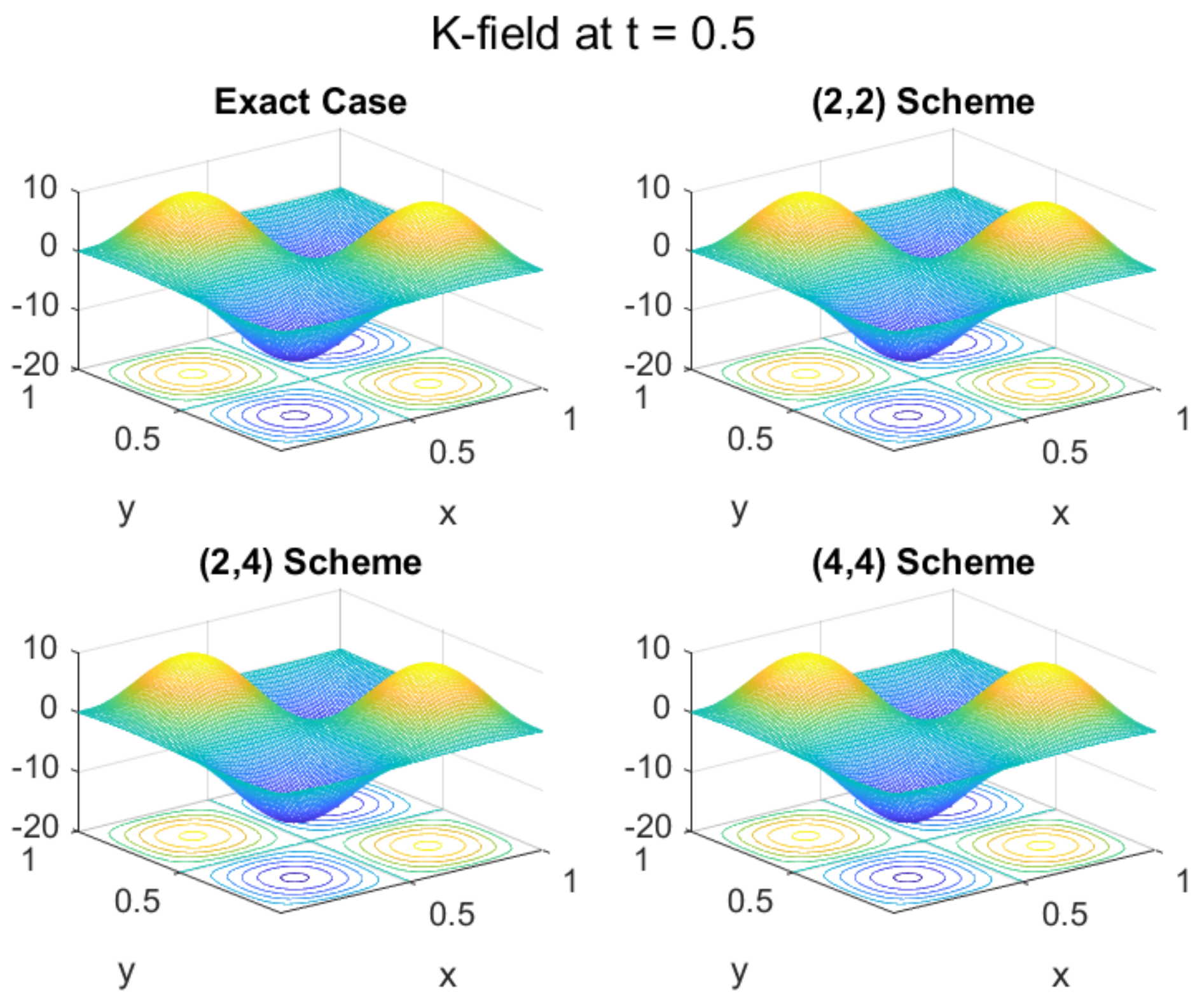}
	  	\caption{Exact and numerical solutions for Maxwell-Drude metamaterial in 2D at $t = 0.5$ for $K$.}	
	  	\label{fig3}
	  \end{figure}
   The quiver plots of the electric field $(E_x,E_y)$ are shown in Figure \ref{fig4}. 
   	   \begin{figure}[h!]
   	\centering
   	\includegraphics[scale= 0.48]{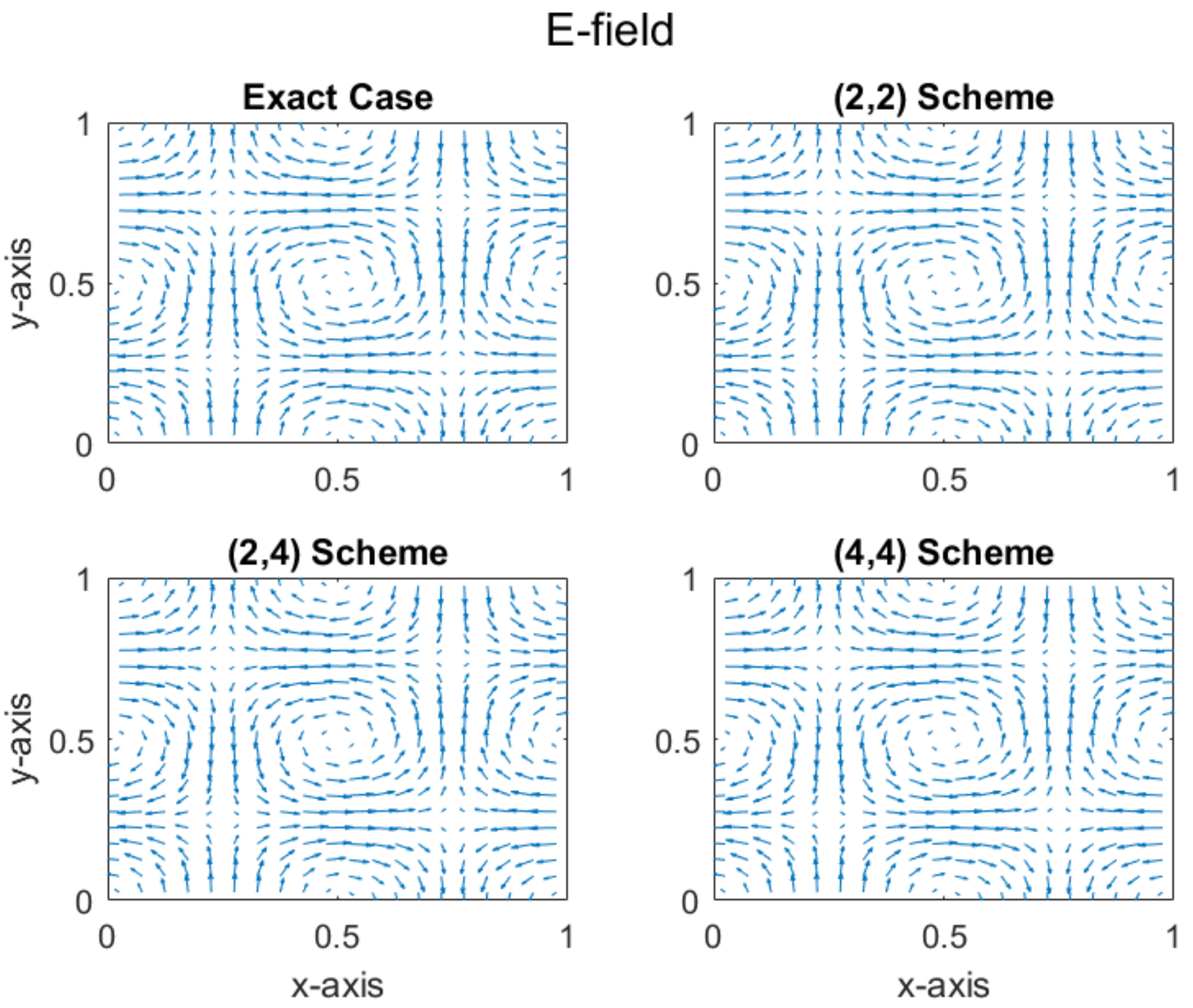} 
   	\caption{Quiver plot of the exact and numerical solutions for Maxwell-Drude metamaterial in 2D. Here we chose $N=100$. For visual purposes, we adjust both components localization to the mesh grid center: $E_x$ and $E_y$ are computed on different grids.}	
   	\label{fig4}
   \end{figure}
Figure \ref{fig10}, Tables \ref{Table4} and \ref{Table5} provide the rates of convergence for the three schemes. Similar conclusions to the 1D case hold.   
 \begin{figure}[h!]
 	\centering
 	\includegraphics[scale= 0.19]{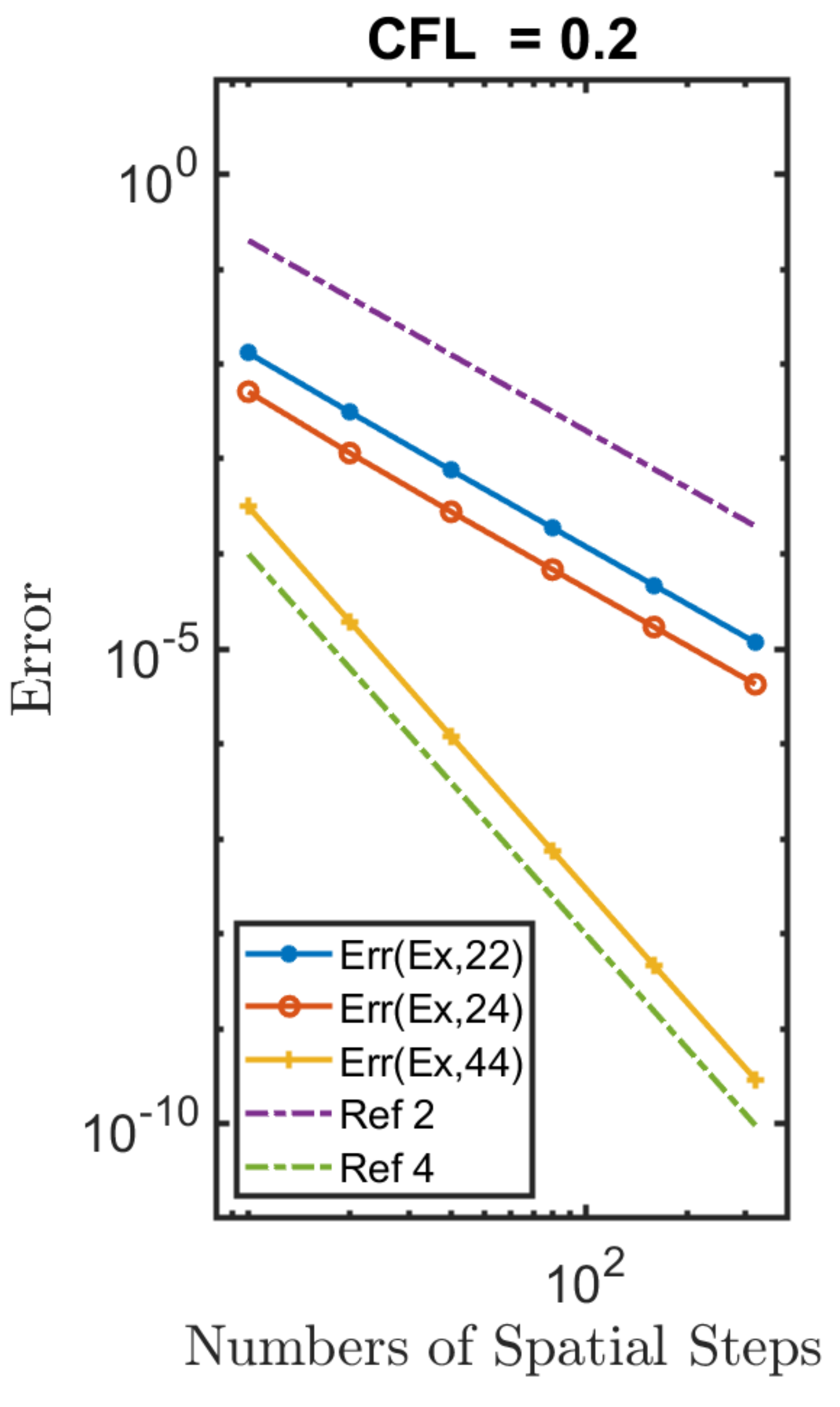} 
 	\includegraphics[scale= 0.19]{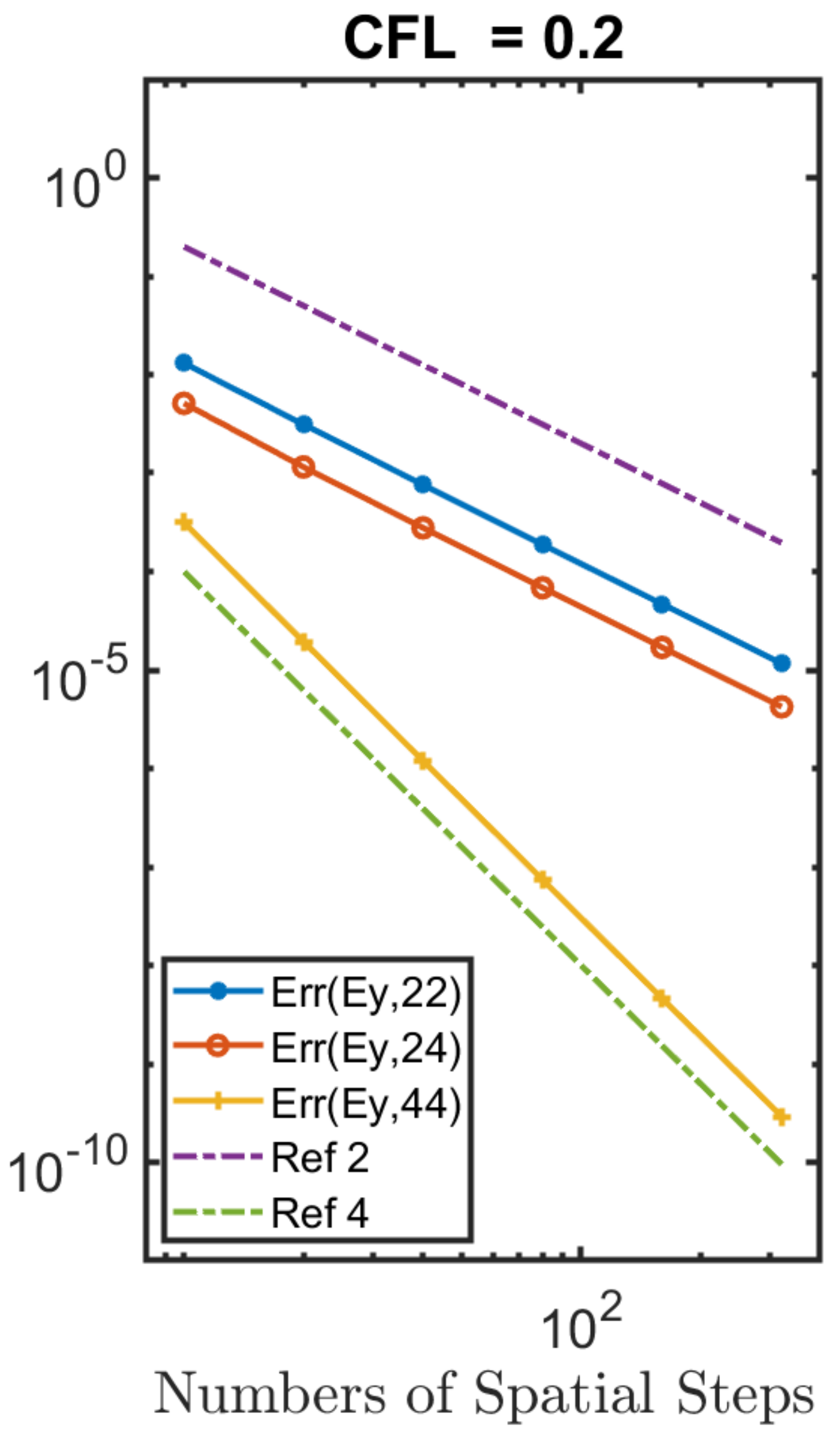} 
 	\includegraphics[scale= 0.19]{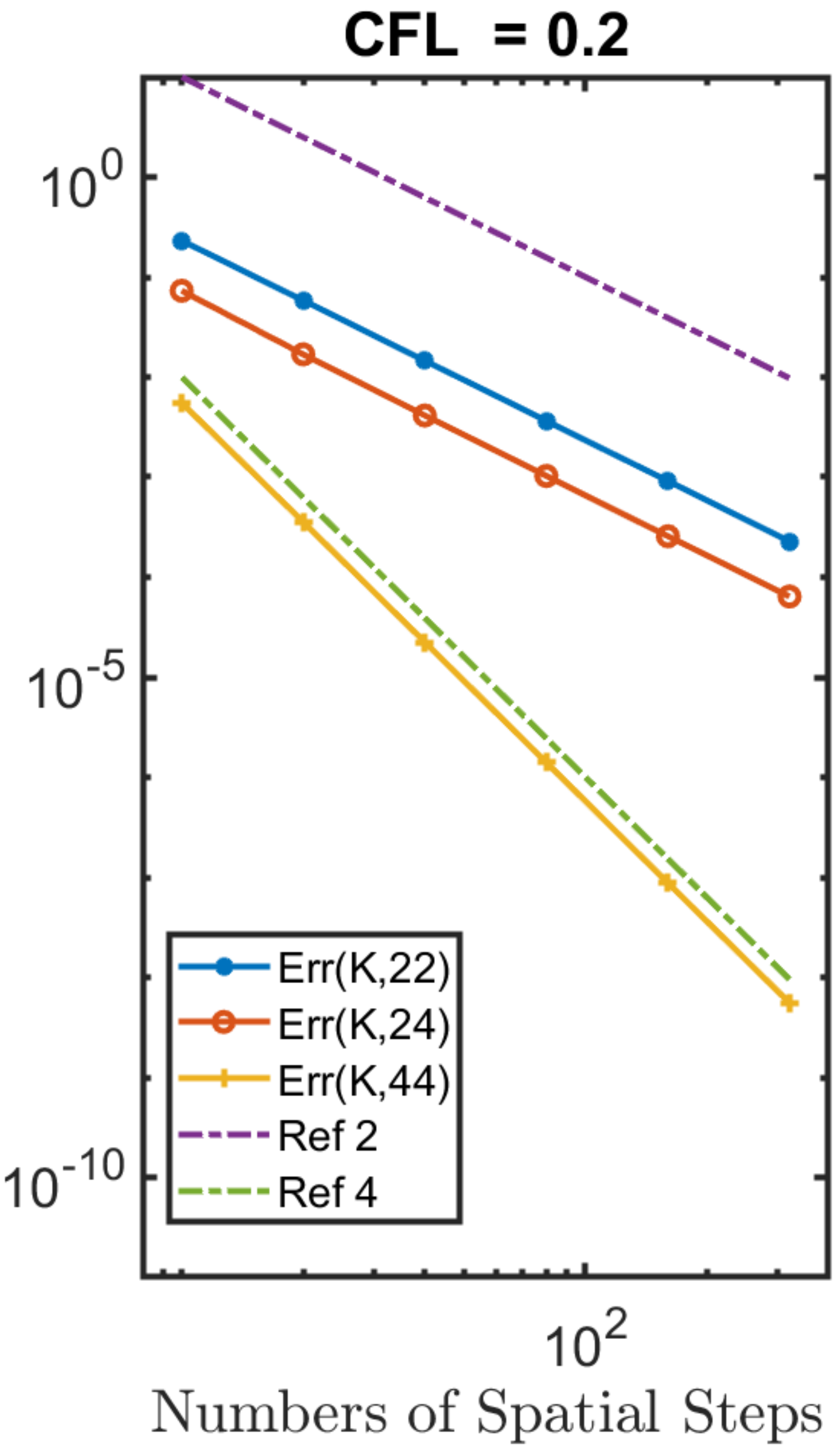}
 	\caption{Convergence rates for the Maxwell-Drude metamaterial in 2D {Left: Err($E_x$), Middle: Err($E_y$), and Right:  Err($K$)}.}	
 	\label{fig10}
 \end{figure}

	\begin{table}[h!]
	\centering
	\begin{tabular}{ |C{2cm}|C{2cm}|C{2cm}|C{2cm}|C{2cm}|C{2cm}| }
		\hline
		\multicolumn{6}{|c|}{ Error for \textbf{E} $ =(E_x,E_y)$ with CFL $\nu = 0.2$} \T \B  \\
		\hline
		$\Delta t$ & $\Delta x$   & \multicolumn{4}{c|}{(4,4)-scheme} 
		\T \B   \\
		\cline{3-6}
		& & \textbf{{Err($E_x$)}}  & \textbf{Rate} & \textbf{{Err($E_y$)}}  & \textbf{Rate} \\
		\hline		 
		2.00e-2 & 1.00e-1    
		& 3.136e-04      & -
		& 3.136e-04      & - \\
		\hline 	
		1.00e-2 & 5.00e-2    
		& 1.943e-05      & -4.013
		& 1.943e-05      & -4.013 \\
		\hline 	
		5.00e-3 & 2.50e-2    
		& 1.197e-06      & -4.021
		& 1.197e-06      & -4.021 \\
		\hline 
		2.50e-3 & 1.25e-2    
		& 7.406e-08      & -4.014
		& 7.406e-08      & -4.014 \\
		\hline 	
		1.25e-3 & 6.25e-3    
		& 4.605e-09      & -4.008
		& 4.605e-09      & -4.008 \\
		\hline 	
		6.25e-4 & 3.125e-3    
		& 2.919e-10      & -3.980
		& 2.919e-10      & -3.980 \\
		\hline 	
		\hline
		$\Delta t$ & $\Delta x$   & \multicolumn{4}{c|}{(2,4)-scheme} 
		\T \B   \\
		\cline{3-6}
		& & \textbf{{Err($E_x$)}}  & \textbf{Rate} & \textbf{{Err($E_y$)}}  & \textbf{Rate} \\
		\hline		 
		2.00e-2 & 1.00e-1    
		& 5.107e-03      & -
		& 5.107e-03      & - \\
		\hline 	
		1.00e-2 & 5.00e-2    
		& 1.154e-03      & -2.146
		& 1.154e-03      & -2.146       \\
		\hline 	
		5.00e-3 & 2.50e-2    
		& 2.776e-04      & -2.055
		& 2.776e-04      & -2.055      \\
		\hline 
		2.50e-3 & 1.25e-2    
		& 6.831e-05      & -2.023
		& 6.831e-05      & -2.023      \\
		\hline 	
		1.25e-3 & 6.25e-3    
		& 1.696e-05      & -2.010
		& 1.696e-05      & -2.010      \\
		\hline 	
		6.25e-4 & 3.125e-3    
		& 4.225e-06      & -2.005
		& 4.225e-06      & -2.005     \\
		\hline 	
		\hline
		$\Delta t$ & $\Delta x$   & \multicolumn{4}{c|}{(2,2)-scheme} 
		\T \B   \\
		\cline{3-6}
		& & \textbf{{Err($E_x$)}}  & \textbf{Rate} & \textbf{{Err($E_y$)}}  & \textbf{Rate} \\
		\hline		 
		2.00e-2 & 1.00e-1    
		& 1.307e-02      & -
		& 1.307e-02      & - \\
		\hline 	
		1.00e-2 & 5.00e-2    
		& 3.117e-03      & -2.068
		& 3.117e-03      & -2.068     \\
		\hline 	
		5.00e-3 & 2.50e-2    
		& 7.607e-04      & -2.035
		& 7.607e-04      & -2.035      \\
		\hline 
		2.50e-3 & 1.25e-2    
		& 1.879e-04      & -2.018
		& 1.879e-04      & -2.018      \\
		\hline 	
		1.25e-3 & 6.25e-3    
		& 4.667e-05      & -2.009
		& 4.667e-05      & -2.009     \\
		\hline 	
		6.25e-4 & 3.125e-3    
		& 1.163e-05      & -2.005
		& 1.163e-05      & -2.005      \\
		\hline 		
				    						    		
	\end{tabular}  
	\caption{Errors of $E$ for the Maxwell-Drude metamaterial in 2D.}
	\label{Table4}
		 \vspace{-0.5cm}
\end{table}	
\begin{table}[h!]
	\centering
	\begin{tabular}{ |C{2cm}|C{2cm}|C{2cm}|C{2cm}| }
		\hline
		\multicolumn{4}{|c|}{Error for $K$ with CFL $\nu = 0.2$} \T \B  \\
		\hline
		$\Delta t$ & $\Delta x$   & \multicolumn{2}{c|}{(4,4)-scheme}   \T \B   \\
		\cline{3-4}
		& & \textbf{{Err($K$)}}  & \textbf{Rate} \\
		\hline		 
		2.00e-2 & 1.00e-1    
		& 5.525e-03      & - \\
		\hline    
		1.00e-2 & 5.00e-2
		& 3.556e-04      & -3.958 \\
		\hline    
		5.00e-3 & 2.50e-2    
		& 2.238e-05      & -3.990\\ 
		\hline
		2.50e-3 & 1.25e-2    
		& 1.402e-06     & -4.000\\
		\hline
		1.25e-3 & 6.25e-3    
		& 8.766e-08		& -3.999 \\
		\hline
		6.25e-4 & 3.125e-3    
		& 5.549e-09		&  -3.982\\
		\hline	
		\hline		
		$\Delta t$ & $\Delta x$   & \multicolumn{2}{c|}{(2,4)-scheme}   \T \B   \\
		\cline{3-4}
		& & \textbf{{Err($K$)}}  & \textbf{Rate} \\
		\hline		 
		2.00e-2 & 1.00e-1    
		& 7.302e-02      & - \\
		\hline    
		1.00e-2 & 5.00e-2
		& 1.688e-02       & -2.113 \\
		\hline    
		5.00e-3 & 2.50e-2    
		& 4.132e-03      & -2.031 \\ 
		\hline
		2.50e-3 & 1.25e-2    
		& 1.028e-03     & -2.008 \\
		\hline
		1.25e-3 & 6.25e-3    
		& 2.565e-04		& -2.002 \\
		\hline
		6.25e-4 & 3.125e-3    
		& 6.411e-05		&  -2.001 \\
		\hline						
		\hline		
		$\Delta t$ & $\Delta x$   & \multicolumn{2}{c|}{(2,2)-scheme}   \T \B   \\
		\cline{3-4}
		& & \textbf{{Err($K$)}}  & \textbf{Rate} \\
		\hline		 
		2.00e-2 & 1.00e-1    
		& 2.310e-01      & - \\
		\hline    
		1.00e-2 & 5.00e-2
		& 5.830e-02       & -1.987 \\
		\hline    
		5.00e-3 & 2.50e-2    
		& 1.460e-02      & -1.998 \\ 
		\hline
		2.50e-3 & 1.25e-2    
		& 3.651e-03     & -1.999 \\
		\hline
		1.25e-3 & 6.25e-3    
		& 9.129e-04		& -2.000 \\
		\hline
		6.25e-4 & 3.125e-3    
		& 2.282e-04		&  -2.000 \\
		\hline			
	\end{tabular} 
	\caption{Errors of $K$ for the Maxwell-Drude metamaterial in 2D.}
	\label{Table5}
			 \vspace{-0.5cm}		
\end{table}	

\section{Conclusion}\label{sec:conclu}
We have constructed a full fourth-order FDM scheme for the time dependent Maxwell-Drude metamaterial model, established stability of the continuous model, by proving energy estimates for the semi and fully-discrete schemes in one dimension. {The new scheme is conditionally stable as the associated discrete energy mimics the one from the continuous model under a CFL condition.} Theoretical rates of convergence have been validated by numerical examples in one and two dimensions, and comparison of this new scheme with the (2,2)-FDM scheme and (2,4)-FDM scheme have been provided. Similar discrete energy estimates hold for two and three dimensions. Extensions to include the divergence conditions will be considered in future work. Current work involves developing the new FDM method for other metamaterial models, such as the Drude model with dissipation and the Lorentz metamaterial model. In these cases the electromagnetic fields are completely coupled, and there is an energy decay due to dissipation. Extensions of our method to metamaterial-dielectric interface problems are also currently being developed.

\section*{Acknowledgment}
This material is based upon work partially supported by the National Science Foundation under the grant numbers: NSF-1720116 (PI Bokil), and NSF-1819052 (PI Carvalho). 


%
%
%


\begin{thebibliography}{1}

\bibitem{akselrod2014probing}
{\sc G.~M.~Akselrod, C.~Argyropoulos, T.~B.~Hoang, C.~Cirac{\`\i},
  C.~Fang, J.~Huang, D.~R.~Smith and M.~H.~Mikkelsen},
{\em Probing the mechanisms of large {P}urcell enhancement in plasmonic
  nanoantennas,} Nature Photonics, 8 (2014), pp.~835--840.

\bibitem{angel2019}
{\sc J.~B. Angel, J. W. Banks, W.~D.~Henshaw, M.~J.~Jenkinson, A.~V.~Kildishev, G.~Kovačič and D.~W.~Schwendeman}, {\em A high-order accurate scheme for Maxwell's equations with a generalized dispersive material model}, Journal of Computational Physics, 378 (2019), pp.~411--444.

\bibitem{anne2000}
{\sc L.~Ann{\'e}, P.~Joly and Q.~H. Tran}, {\em Construction and analysis of
  higher order finite difference schemes for the 1d wave equation},
  Computational Geosciences, 4 (2000), pp.~207--249.
  
  \bibitem{BaDeEb03}
 { \sc W.~L.~Barnes, A.~Dereux and T.~W.~Ebbesen,}
{ \em Surface plasmon subwavelength optics,}
Nature, 424 (2003), pp.~824--830.

\bibitem{bokil2011}
{\sc V.~Bokil and N.~Gibson}, {\em Analysis of spatial high-order finite
  difference methods for Maxwell's equations in dispersive media}, IMA Journal
  of Numerical Analysis, 32 (2011), pp.~926--956.

\bibitem{bokil2014}
  {\sc V.~Bokil and N.~Gibson}, {\em Convergence analysis of Yee schemes for Maxwell's equations in Debye and Lorentz dispersive media},
  International Journal of Numerical Analysis \& Modeling, 11(4), 2014, pp.~657--687. 
\bibitem{cai2007} 
{ \sc W.~Cai, U.~K.~Chettiar, A.~V.~Kildishev and V.~M.~Shalaev,} {\em Optical cloaking with metamaterials,} Nature photonics 4 (2007), pp.~224.

\bibitem{JoCo}
{\sc G.~C.~Cohen and P.~Joly}, {\em Construction analysis of fourth-order finite
  difference schemes for the acoustic wave equation in nonhomogeneous media},
  SIAM Journal on Numerical Analysis, 33 (1996), pp.~1266--1302.

\bibitem{cohenbook}
{\sc G.~C.~Cohen}, {\em Higher-order numerical methods for transient wave
  equations}, ASA, 2003.
  
  \bibitem{GrBo10}
  { \sc D.~K.~Gramotnev and S.~I.~Bozhevolnyi,}
{ \em Plasmonics beyond the diffraction limit,}
 Nature Photonics, 4 (2010), pp.~83--91.

\bibitem{jichun}
{\sc J.~Li and Y.~Huang}, {\em Time-domain finite element methods for Maxwell's
  equations in metamaterials}, vol.~43, Springer Science \& Business Media,
  2012.
  
\bibitem{Maier07} 
{\sc S.~A.~Maier}, {\em Plasmonics: {F}undamentals and {A}pplications}, Springer, 2007.

\bibitem{mittra2008}
{\sc Y.~Hao and R.~Mittra}, { \em FDTD modeling of metamaterials: Theory and applications,} Artech house, 2008.

\bibitem{pekmezci2014}
{\sc A.~Pekmezci and L.~Sevgi}, {\em FDTD-Based Metamaterial (MTM) Modeling and Simulation,} IEEE Antennas and Propagation Magazine, 5 (2014), pp.~289--303.

\bibitem{sannomiya2008situ}
{\sc T.~Sannomiya, C.~Hafner and J.~Voros,}
{\em In situ sensing of single binding events by localized surface plasmon
  resonance,}
 Nano Letters, 8 (2008), pp.~3450--3455.

 \bibitem{yee1966numerical}
   {\sc K.~Yee}, {\em   Numerical solution of initial boundary value problems involving Maxwell's equations in isotropic media}, IEEE Transactions on Antennas and Propagation, 14(3), (1966), pp.~302--307.

\bibitem{young1996}
{\sc J.~L.~Young}, {\em A higher order fdtd method for em propagation in a
  collisionless cold plasma}, IEEE Transactions on Antennas and Propagation, 44 (1996), pp.~1283--1289.
  
  \bibitem{ZaSM05}
  { \sc A.~V.~Zayats, I.~I.~Smolyaninov and A.~A.~Maradudin,}
  {\em Nano-optics of surface plasmon polaritons,}
  Physical Reports, 408 (2005), pp.~131--314.
  
\bibitem{ziol2001}
  {\sc R.~W.~Ziolkowski and E.~Heyman}, {\em Wave propagation in media having negative permittivity and permeability}, Physical Review E, 64 (2001), 056625.
  
  
\end{thebibliography}
\end{document}